\documentclass[11pt,fleqn,leqno]{siamart220329}

\usepackage{amsmath}
\usepackage{amssymb}
\usepackage{graphicx}
\usepackage{url}

\newtheorem{remark}[theorem]{Remark}
\newtheorem{example}[theorem]{Example}
\newtheorem{summary}[theorem]{Summary}

\newcommand{\R}{\mathbb R}
\newcommand{\C}{\mathbb C}
\newcommand{\ph}{\phantom}

\newcommand{\smath}[1]{{\mbox{\scriptsize $#1$}}}

\newcommand{\smtxa}[2]{
{\mbox{\scriptsize
$\left[\!\!
\begin{array}{#1}
#2
\end{array} \!\! \right]$}}}

\catcode`@=11 
\font\tenex=cmex10 
\newdimen\p@renwd
\setbox0=\hbox{\tenex B} \p@renwd=\wd0 
\def\bmat#1{\begingroup \m@th
  \setbox\z@\vbox{\def\cr{\crcr\noalign{\kern2\p@\global\let\cr\endline}}%
    \ialign{$##$\hfil\kern2\p@\kern\p@renwd&\thinspace\hfil$##$\hfil
      &&\quad\hfil$##$\hfil\crcr
      \omit\strut\hfil\crcr\noalign{\kern-\baselineskip}%
      #1\crcr\omit\strut\cr}}%
  \setbox\tw@\vbox{\unvcopy\z@\global\setbox\@ne\lastbox}%
  \setbox\tw@\hbox{\unhbox\@ne\unskip\global\setbox\@ne\lastbox}%
  \setbox\tw@\hbox{$\kern\wd\@ne\kern-\p@renwd\left[\kern-\wd\@ne
    \global\setbox\@ne\vbox{\box\@ne\kern2\p@}%
    \vcenter{\kern-\ht\@ne\unvbox\z@\kern-\baselineskip}\,\right]$}%
  \null\;\vbox{\kern\ht\@ne\box\tw@}\endgroup}
\catcode`@=12    

\author{Michiel~E.~Hochstenbach\thanks{%
Version \today.
Department of Mathematics and Computer Science,
TU Eindhoven,
PO Box 513, 5600 MB, The Netherlands,
{\tt www.win.tue.nl/$\sim$hochsten}.
This author has been supported by an NWO Vidi research grant.}
\and
Christian Mehl\thanks{%
Institut f\"{u}r Mathematik, Technische Universit\"at Berlin, Sekretariat MA 4-5,
Stra\ss e des 17.~Juni 136, 10623 Berlin, Germany,
{\tt mehl@math.tu-berlin.de}.
This author has been supported by a Dutch 4TU AMI visitor's grant.}
\and
Bor~Plestenjak\thanks{%
Faculty of Mathematics and Physics, University of Ljubljana,
Jadranska 19, SI-1000 Ljubljana, Slovenia,
{\tt bor.plestenjak@fmf.uni-lj.si}.
This author
has been supported by the Slovenian Research Agency (grant N1-0154).}}

\title{Solving singular generalized eigenvalue problems \\
Part II: projection and augmentation.}

\begin{document}
\maketitle

\begin{abstract}
Generalized eigenvalue problems involving a singular pencil may be very challenging to solve, both with respect to accuracy and efficiency.
While Part I presented a rank-completing addition to a singular pencil, we now develop two alternative methods. The first technique is based on a projection onto subspaces with dimension equal to the normal rank of the pencil while the second approach exploits an augmented matrix pencil.
The projection approach seems to be the most attractive version for generic singular pencils because of its efficiency, while the augmented pencil approach may be suitable for applications where a linear system with the augmented pencil can be solved efficiently.
\end{abstract}



\pagestyle{myheadings}
\thispagestyle{plain}
\markboth{HOCHSTENBACH, MEHL, AND PLESTENJAK}{SOLVING SINGULAR GEPs BY PROJECTION}

\section{Introduction}
\label{sec:intro}
In this paper, which is a sequel to Part I \cite{HMP19}, we further study the computation of eigenvalues of {\em singular} matrix pencils.
Whereas in \cite{HMP19} a method based on a rank-completing perturbation has been introduced (i.e., an update by a pencil of a rank that is precisely sufficient to render the updated pencil regular), we propose in this paper a scheme based
on rank projection (i.e., a
projection of the pencil onto a subspace of maximal possible size such that the
projected pencil is generically regular), as well as an augmentation method, as two alternative techniques.

We briefly recall the context and main results from \cite{HMP19}.
Let $A-\lambda B$ be a singular pencil, where $A,B$ are (real or complex) $n \times m$ matrices.
This means that either $m \ne n$, or $m = n$ and $\det(A-\lambda B) \equiv 0$.
Then the normal rank of $A-\lambda B$, which is a crucial quantity for the methods in \cite{HMP19} and in this paper, is given by
\[
\textrm{nrank}(A,B) := \max_{\zeta \in \C} \textrm{rank}(A-\zeta B).
\]
The value $\lambda_0\in \mathbb C$ is an eigenvalue of the pencil $A-\lambda B$ when $\textrm{rank}(A-\lambda_0 B) < \textrm{nrank}(A,B)$, and if $\textrm{rank}(B) < \textrm{nrank}(A,B)$ then $\lambda_0=\infty$ is an eigenvalue
of $A-\lambda B$.
Most of this paper focuses on the square case $m=n$. Some results, such as \Cref{main} and Algorithm~1, are also for the nonsquare case, where we can assume $n \ge m$ without loss of generality because we can always switch to the transpose pencil.

In \cite{HMP19}, the predecessor of this paper, we have investigated generic \emph{rank-completing perturbations} of
square singular pencils $A-\lambda B$, i.e., perturbations
of rank $k:=n-\textrm{nrank}(A,B)$ that have the form
\begin{equation}
\label{pert}
\widetilde A-\lambda \widetilde B := A-\lambda B+\tau \, (UD_AV^*-\lambda \, UD_BV^*),
\end{equation}
where $D_A,D_B\in \C^{k,k}$ are diagonal matrices such that $D_A-\lambda D_B$ is a regular
pencil, $U,V\in \C^{n,k}$, and $\tau\in \mathbb C$ is nonzero.
The name \emph{rank-completing perturbation} is based on the
observation that the rank $k$ of the perturbation pencil
$UD_AV^*-\lambda \, UD_BV^*$ is just large enough to ``complete'' the normal rank $n-k$ of $A-\lambda B$ to full rank $n$. Indeed, it turns out that generically with respect to the
entries of $U$ and $V^*$ the perturbed pencil in~\eqref{pert} has full rank and is consequently regular. Moreover,
for each $\tau\ne 0$ the eigenvalues of \eqref{pert}
contain all eigenvalues of the initial singular pencil $A-\lambda B$; see \cite{HMP19} and below for details.
Although the approach is theoretically independent of $\tau \ne 0$, for numerical stability it is recommended to avoid very small or large values;
a suggested value in \cite{HMP19} for $A$ and $B$ scaled so
that $\|A\|_1=\|B\|_1=1$ is $\tau=10^{-2}$.

Rank-completing perturbations have shown to often give excellent results, as seen in \cite{HMP19} as well as subsequent publications
(see, e.g., \cite{KreGlib22}).
Still, the following observations may be considered minor disadvantages of the technique of using rank-completing perturbations.
\begin{itemize}
\item Computing the updates $\tau \, U D_A V^*$ and $\tau \, U D_B V^*$ in \eqref{pert} requires
$\mathcal{O}(kn^2)$ operations as extra work, and potentially leads to a minor loss of accuracy.
When $k = \mathcal{O}(n)$, as is the case in problems related to systems
of bivariate polynomials in \cref{sec:poly}, the amount of
extra work is $\mathcal{O}(n^3)$, albeit with a small constant.
This means that this amount is of the same order as that for solving the generalized eigenproblem.
\item The approach is not parameter-free, in the sense that generic (e.g., random) $n \times k$ matrices $U$ and $V$
as well as a parameter $\tau$ and $k \times k$ (diagonal) matrices $D_A$ and $D_B$ have to be selected.
In contrast, the approach developed in \cref{sec:proj} and \cref{subsec:proj} does not require $\tau$, $D_A$, and $D_B$,
while the technique in \cref{sec:augm} avoids any work prior to the solution of the generalized eigenproblem and
leaves the matrices $A$ and $B$ intact.
\item Although a singular pencil has only $r < n$ true eigenvalues (where sometimes even $r \ll n$), the method still needs to compute eigenvalues of an $n \times n$ pencil together with the right and left eigenvectors; this forms the main computational work of the algorithm.
The projection approach in \cref{sec:proj} and \cref{subsec:proj} works on a smaller pencil of size equal to the normal rank: $(n-k)\times (n-k)$.
\end{itemize}
Motivated by these observations, we present two alternative approaches in \cref{subsec:proj} (with its theoretical foundation developed in \cref{sec:proj}) and \cref{sec:augm}.
These techniques have their own features, resulting in alternative and complementing schemes.
The approaches are partly inspired by techniques for constrained eigenvalue problems;
see, e.g., the classic work by Golub, Arbenz, and Gander \cite{Gol73, AGG88, Arb12}
for symmetric constrained standard and generalized eigenvalue problems.
In this paper, the focus is on generalized nonsymmetric constrained eigenproblems
(see also \cite[Remark~4.8]{HMP19}), with special attention to the singularity of the pencil.

The paper is organized as follows. In \cref{sec:rankcomplete} we recall the main results of Part I \cite{HMP19}, to ensure that this document is reasonably self-contained, but also
to provide a preparation for new results derived in this paper.
Several necessary technical results are presented in \cref{sec:prelim}
before the theory of generic projections of singular pencils to a regular
pencil of size equal to the normal rank is developed in \cref{sec:proj}. The proof of the main result is delayed to \cref{sec:proof} for an easy flow of the paper.
 \Cref{subsec:proj} then introduces a projection approach based on
the results from the previous section. This method seems attractive for all problems, particularly for those having large $k/n$ ratios, that is, a relatively small normal rank compared to the size of the pencil.
As a second alternative method, \cref{sec:augm} presents a new method based on an augmented pencil, which may be interesting in cases where the
ratio $k/n$ is small (i.e., the normal rank is large in comparison to $n$) and we can solve a (regular) eigenvalue problem with the augmented matrices efficiently. In \cref{sec:classif} we discuss how to extract finite eigenvalues from eigenvalues of the regular part.
Some numerical experiments are presented in \cref{sec:num},
followed by \cref{sec:nrank}, where we discuss the computation of the normal rank and consequences of using an inaccurate normal rank.
Finally, conclusions are summarized in \cref{sec:concl}.

Throughout the paper, $\|\cdot\|$ denotes the 2-norm.

\section{Review of rank-completing perturbations of singular pencils}
\label{sec:rankcomplete}
The following main result from Part~I \cite[Summary~4.7]{HMP19} (see also
\cite[Thms.~4.3 and 4.6]{HMP19}) characterizes the dependence of eigenvalues and eigenvectors of the perturbed
pencil~\eqref{pert} on $\tau$, $D_A$, $D_B$, $U$, and $V^*$.
(For a reminder of the Kronecker canonical form and a specification of the term ``generic'' we refer to section 3.)

\begin{summary} \rm \label{summry}
Let $A-\lambda B$ be an $n\times n$ singular pencil of normal rank $n-k$
with left minimal indices $n_1,\dots,n_k$ and right minimal indices $m_1,\dots,m_k$.
Let $N = n_1 + \cdots + n_k$ and $M = m_1 + \cdots + m_k$. Thus,
the regular part of $A-\lambda B$ has size $r:=n-N-M-k$. Then,
generically with respect to the entries of the matrices $U$ and $V^*$ in~\eqref{pert},
the perturbed pencil \eqref{pert} is regular and its eigenvalues are classified as follows:
\begin{enumerate}
\item \emph{True eigenvalues}: There are $r$ eigenvalues that coincide precisely with
the eigenvalues of the original pencil $A-\lambda B$.
The corresponding right eigenvectors $x$ and left eigenvectors $y$ satisfy the
orthogonality relations $V^*x=0$ and $U^*y=0$.
\item \emph{Prescribed eigenvalues}: There are $k$ eigenvalues such that the
corresponding right eigenvectors $x$ and left eigenvectors $y$ satisfy
both $V^*x\ne 0$ and $U^*y\ne 0$. These eigenvalues are the $k$ eigenvalues
of $D_A-\lambda D_B$.
\item \emph{Random eigenvalues}: These are the remaining $N+M$ ($= n-r-k$) eigenvalues.
They are simple and if $\mu$ is such an eigenvalue with the corresponding right eigenvector $x$ and
left eigenvector $y$,
then we either have $V^*x=0$ and $U^*y\ne 0$, or $V^*x\ne 0$ and $U^*y=0$.
\end{enumerate}
\end{summary}

Note that the vectors $x$ and $y$ referred to above are eigenvectors
of the {\em perturbed} pencil.
They may be viewed as well-defined eigenvectors of the original pencil with two added orthogonality constraints $V^*x=0$ and $U^*y=0$. This agrees with a recent definition of eigenvectors of singular pencils in \cite{DopNof}, where eigenvectors are subspaces of certain quotient spaces, as the vectors $x$ and $y$ as above are particular elements of these
quotient spaces.

In our three approaches (the rank-completing algorithm from \cite{HMP19} and the new algorithms from
\cref{subsec:proj} and \cref{sec:augm}) we first extract the true eigenvalues $\lambda_i$, $i=1,\ldots,r$, together with the right and left eigenvectors $x_i$ and $y_i$ in the sense described above, and then we use the values $|y_i^*Bx_i|$ to further divide the true eigenvalues into finite and infinite ones; for details see \cref{sec:classif}.

In \cref{tab:work} we provide an overview of the main work for the three approaches: the perturbation method of \cite{HMP19}, and the new projection and augmentation schemes of this paper.
The numbers of various types of eigenvalues of the different methods are given in \cref{tab:nreigs}.

\begin{table}[htb!] \label{tab:work}
\centering
\caption{Main work for an $n \times n$ singular pencil with normal rank $n-k$ (and therefore rank-completing dimension $k$) for each of the three methods: determining the normal rank, the update/projection/augmentation step, and the actual solution of the eigenproblem.}
{\footnotesize
\begin{tabular}{lcccll} \hline \rule{0pt}{2.3ex}%
Method $\backslash$ Task & nrank & Preparation & Eigentriplets & Attractive for \\ \hline \rule{0pt}{2.6ex}%
Perturbation by rank $k$ & ${\cal O}(n^3)$ & ${\cal O}(n^2k)$ & ${\cal O}(n^3)$ & (Method of \cite{HMP19}) \\[1mm]
Projection onto dim.~$n-k$ & ${\cal O}(n^3)$ & ${\cal O}(n^2(n-k))$ & ${\cal O}((n-k)^3)$ & Always, especially for large $k$ \\[1mm]
Augmentation by $k$ & ${\cal O}(n^3)$ & $0$ & ${\cal O}((n+k)^3)$ & For large $n$ if we can solve the \\
\quad &&&& augmented system efficiently \\ \hline
\end{tabular}}
\end{table}

\begin{table}[htb!] \label{tab:nreigs}
\centering
\caption{Number of various types of eigenvalues for an $n \times n$ singular pencil with normal rank $n-k$: true (finite $n_{\rm f}$ and infinite $n_{\rm i}$), prescribed, and random $n_{\rm r}$, where $n_{\rm f} + n_{\rm i} + n_{\rm r} = n-k$.}
{\footnotesize
\begin{tabular}{lcccc} \hline \rule{0pt}{2.3ex}%
Method & True $\lambda$'s & Prescribed $\lambda$'s & Random $\lambda$'s & Size \\ \hline \rule{0pt}{2.5ex}%
Perturbation & $n_{\rm f}+n_{\rm i}$ & $k$ & $n_{\rm r}$ & $n$ \\
Projection & $n_{\rm f}+n_{\rm i}$ & $0$ & $n_{\rm r}$ & $n-k$ \\
Augmentation & $n_{\rm f}+n_{\rm i}$ & $2k$ & $n_{\rm r}$ & $n+k$ \\ \hline
\end{tabular}}
\end{table}

\begin{example}\rm
We consider a system of two random bivariate polynomials $p_1$ and $p_2$ of total degree 10 (e.g., coefficients that follow a normal distribution).
The number of (finite) roots is (generically) 100.
The determinantal representation from \cite{BDD17} gives
$19 \times 19$ matrices $A_i, B_i, C_i$ with
$p_i(\lambda,\mu) = \det(A_i+\lambda B_i+\mu C_i)$ for $i=1,2$.
To solve for $\lambda$, we define the operator determinants
$\Delta_1 = C_1 \otimes A_2 - A_1 \otimes C_2$ and
$\Delta_0 = B_1 \otimes C_2 - C_1 \otimes B_2$,
which are of size $361 \times 361$, with ranks $280$ and $198$, respectively.
The normal rank is also 280, which means that the rank-completing dimension
is $k = 361-280 = 81$. Therefore, the perturbed pencil \eqref{pert}, modified by a rank-81 perturbation as in \cite{HMP19}, has 81 prescribed eigenvalues,
which we select ourselves by the particular choice of the update pencil $(D_A,D_B)$.
The perturbed pencil for this type of application has
no random eigenvalues (so $n_{\rm r} = 0$) and all 280 remaining eigenvalues of
$(\widetilde A, \widetilde B)$ are true eigenvalues of the original pencil.
Of these, 100 are finite true eigenvalues, corresponding to roots of the polynomial system; i.e., $n_{\rm f} = 100$, $n_{\rm i} = 180$.
The projection method developed in \cref{sec:proj} and \cref{subsec:proj} is more efficient for this example as it computes eigenvalues of a $280\times 280$ pencil, while the rank-completing
method from \cite{HMP19} works with a larger $361\times 361$ pencil; the difference is due to $k=81$ prescribed eigenvalues that are needed in the rank-completing method.
In \cref{sec:poly} we consider this application in more detail.
\end{example}

\section{Preliminaries}
\label{sec:prelim}
In this section, we gather some results that will be used in the rest of the paper. This section has a modest overlap with \cite{HMP19}, which is necessary to explain and prove the new results of this paper.
We start by clarifying the notion of genericity.
Let $\mathbb F$ denote one of the fields $\C$ or $\R$.
A set $\mathcal A\subseteq\mathbb F^m$ is called \emph{algebraic} if it is the set of common zeros of
finitely many polynomials in $m$ variables.
A set $\Omega\subseteq\mathbb F^m$ is called \emph{generic} if its complement is contained in a
proper algebraic subset of $\mathbb F^m$, i.e., an algebraic subset that is not $\mathbb F^m$.

In \cite{HMP19} the following concept of genericity of matrices has been used.
There, a set $\Omega\subseteq\C^{n,k}$ has been called generic if it can be
canonically identified with a generic subset of $\C^{nk}$.
In this paper, we will use a slightly more general concept which is necessary as
sometimes we have to explicitly deal with complex conjugation which is not a complex
polynomial map. Therefore, we view a complex number as the real pair of its real and imaginary parts and identify $\C^m$ canonically with the real
vector space $\R^{2m}$. We thus say that $\Omega\subseteq\C^{n,k}$
is \emph{generic} if it can be canonically identified with a generic subset of
$\R^{2nk}$. Clearly, a set $\Omega\subseteq\C^{n,k}$ that is generic
in the sense of \cite{HMP19} is also generic in the new sense, because any polynomial
in (say) $m$ complex variables can also be viewed as a polynomial in $2m$ real variables
(or more precisely as a pair of real polynomials in $2m$ real variables that are
obtained by taking the real and imaginary parts of the aforementioned polynomial which
still may have complex coefficients). Hence, all results from
\cite{HMP19} can also be expressed in terms of the new notion and we will do so
without further notice when quoting those results.

If $\Omega\subseteq\C^{n,k}$ is generic, then so is the set of all matrices from $\Omega$ that have
full rank. This immediately follows from two basic facts. Firstly, the set of all matrices
from $\C^{n,k}$ having full rank is a generic set, because an $n\times k$ matrix has rank less than $\rho := \max(n,k)$ if and only if all $\rho \times \rho$ minors (which are polynomials in the entries of the matrix, or, more precisely, in the real and imaginary parts of the entries) are zero.
Secondly, the intersection of finitely many generic sets is again a generic set. Therefore, we can always assume in the
following that a generic set $\Omega\subseteq\C^{n,k}$ only consists of matrices having full rank.

Next, we recall the well-known Kronecker canonical form (KCF) for matrix pencils; see, e.g., \cite{Gan59}.

\begin{theorem}[Kronecker canonical form]\label{thm:knf}
Let $A,B\in \C^{n,m}$. Then there exist nonsingular matrices
$P\in \C^{n,n}$ and $Q\in \C^{m,m}$ such that
\begin{equation}\label{eq:knf}
P\,(A-\lambda B)\,Q=\left[\begin{array}{cc} R(\lambda)&0\\ 0 & S(\lambda)\end{array}\right], \qquad
R(\lambda)=\left[\begin{array}{cc}J-\lambda I_r&0\\ 0&I_s-\lambda N\end{array}\right],
\end{equation}
where $J$ and $N$ are in Jordan canonical form with $N$ nilpotent. Furthermore,
\[
S(\lambda)=
{\rm diag}\big(L_{m_1}(\lambda), \dots, L_{m_k}(\lambda), \ L_{n_{1}}(\lambda)^\top, \dots, L_{n_{\ell}}(\lambda)^\top\big),
\]
where
$L_{j}(\lambda)=[0 \ \, I_{j}]-\lambda \, [I_{j} \ \, 0]$
is of size $j\times (j+1)$, and $m_i\ge 0$ for $i=1,\dots,k$,
and $n_i\ge 0$ for $i=1,\dots,\ell$.
\end{theorem}

The pencil $R(\lambda)$ in~\eqref{eq:knf} is called the regular part of $A-\lambda B$ and contains the eigenvalues
of $A-\lambda B$. The block $J-\lambda I_r$ contains the finite eigenvalues of $A-\lambda B$, whereas the block $I_s-\lambda N$ contains the eigenvalue infinity. For a finite eigenvalue $\lambda_0\in\mathbb C$
of geometric multiplicity $m$, we have precisely $m$ Jordan blocks associated with $\lambda_0$ in $J$, denoted by $J_{d_1}(\lambda_0),\dots,J_{d_m}(\lambda_0)$, where $d_1,\dots,d_m$ stand for the sizes of the Jordan blocks. We say that $\lambda_0$ is semisimple if $d_1=\cdots=d_m=1$, and simple if $m=1$ and $d_1=1$.
Analogously, simplicity and semisimplicity of the eigenvalue infinity
are defined via the number and sizes of the Jordan blocks in $N$.
The pencil $S(\lambda)$ is called the singular part of $A-\lambda B$ and contains right
singular blocks $L_{m_1}(\lambda),\dots,L_{m_k}(\lambda)$ and left singular blocks
$L_{n_{1}}(\lambda)^\top,\dots,L_{n_{\ell}}(\lambda)^\top$, where
 $m_1,\dots,m_k$ and
$n_{1},\dots,n_{\ell}$ are called the \emph{right} and \emph{left minimal indices} of the pencil, respectively. (Note that the values $m_i=0$ and $n_i=0$
are explicitly allowed here. In that case the blocks $L_{m_i}(\lambda)$ or $L_{n_i}(\lambda)^\top$ correspond to a zero column or row in
$\mathcal S(\lambda)$, respectively.)
We highlight that the regular and singular parts of $A-\lambda B$ are uniquely
determined (up to permutation of blocks), but the transformation matrices $P$ and $Q$ are not.

We say that a
subspace ${\cal M}$ is a \emph{reducing subspace} \cite{VD83} for the pencil $A-\lambda B$ if ${\rm dim}(A{\cal M} + B{\cal M}) =
{\rm dim}({\cal M}) - k$, where $k$ is the number of right singular blocks.
\emph{The minimal reducing subspace} ${\cal M}_{\rm RS}(A,B)$ is the
intersection of all reducing subspaces and is spanned
by the columns of $Q$ corresponding to the blocks $L_{m_1}(\lambda),\ldots,L_{m_k}(\lambda)$.
In a similar way ${\cal L}$ is a \emph{left reducing subspace} for the pencil $A-\lambda B$ if
${\rm dim}(A^*{\cal L}+B^*{\cal L})={\rm dim}({\cal L}) - \ell$, where $\ell$ is the number of left singular blocks,
and \emph{the minimal left reducing subspace} ${\cal L}_{\rm RS}(A,B)$ is the
intersection of all left reducing subspaces and is spanned by the columns of $P^*$ corresponding to the blocks
$L_{n_{1}}(\lambda)^\top,\dots,L_{n_{\ell}}(\lambda)^\top$.

The main results of~\cite{HMP19} on rank-completing perturbations of the form $A-\lambda B+\tau \, U\,(D_A-\lambda D_B)\,V^*$ of an $n\times n$ pencil of normal rank $n-k$ with matrices $U,V\in \C^{n,k}$ and $D_A,D_B\in \C^{k,k}$
are stated to hold ``generically with respect to the entries of $U$ and $V^*$'' which is equivalent to the
existence of a generic set $\Omega\subseteq\C^{n,k}\times\C^{k,n}$
such that the corresponding statements hold for all $(U,V^*)\in\Omega$.
Due to the different notion of genericity in this paper, it is no longer necessary
to distinguish between the entries of $V$ and $V^*$, because any polynomial in the real and imaginary parts of the entries of $V$ is also a polynomial in the real and
imaginary parts of the entries of $V^*$. Furthermore, it will turn out
to be convenient to express this genericity in a slightly different way, and the next lemma justifies that this is indeed possible.

\begin{lemma}\label{lem:gen}
Let $\Omega\subseteq\C^{n,k}\times\C^{n,k}$ be a generic set. Then there exists
a generic set $\Omega_2\subseteq\C^{n,k}$ with the property that for each $U\in\Omega_2$ there exists
a generic set $\Omega_3\subseteq\C^{n,k}$ such that $V\in\Omega_3$ implies $(U,V)\in\Omega$.
\end{lemma}

\begin{proof}
By definition, the complement of $\Omega$ is contained in the set of common zeros of finitely many polynomials
in $4nk$ real variables. In fact, we may assume without loss of generality, that
the complement of $\Omega$ is contained in the set of zeros of just one
polynomial $p$ by choosing $p$ to be the product of the previous finitely many polynomials.
Then for $U\in \C^{n,k}$ and $T\in \C^{n,k}$
we have that $p(U,T)\ne 0$ implies $(U,T)\in\Omega$. (Here, $p(U,T)$ is interpreted as evaluating $p$ in the real and imaginary parts of the entries of $U$ and $T$.) For the moment, let $\widetilde T\in \C^{n,k}$ be fixed
such that $\widetilde p:\C^{n,k}\to\C$ with $\widetilde p(U):=p(U,\widetilde T)$ is not the zero polynomial.
Then the set $\Omega_2\subseteq\C^{n,k}$ of all $U\in \C^{n,k}$ for which $\widetilde p(U)\ne 0$ is,
by definition, a generic set. Now let $U\in\Omega_2$ be fixed. Then the set $\Omega_3'$ of all $T\in \C^{n,k}$
for which we have $p(U,T)\ne 0$ is a generic set, because by construction of $\Omega_2$ the polynomial
$\widehat p:T\mapsto p(U,T)$ is not the zero polynomial. The statement of the lemma now follows by taking
$\Omega_3\subseteq\C^{n,k}$ to be the set of all $V\in \C^{n,k}$ satisfying $\widehat p(V)\ne 0$ as this implies $p(U,V)\ne 0$.
\end{proof}

The following theorem combines the results from \cite[Thm.~4.3]{HMP19} and \cite[Thm.~4.6]{HMP19}, and adapts the
notion of genericity as outlined above and with the help of \cref{lem:gen}.

\begin{theorem}\label{thm:nfp}
Let $A-\lambda B$ be an $n\times n$ singular pencil of normal rank $n-k$ and let $D_A,D_B\in \C^{k,k}$ be such
that $D_A-\lambda D_B$ is diagonal and regular and all eigenvalues of $D_A-\lambda D_B$ are distinct from the eigenvalues of $A-\lambda B$.
Then there exists a generic set $\Omega_2\subseteq\C^{n,k}$ with the following property: for each $U\in\Omega_2$
there exists a generic set $\Omega_3\subseteq\C^{n,k}$ such that for all $V\in\Omega_3$
the following statements hold:
\begin{itemize}
\item[1)] For each $\tau\ne 0$, there exist nonsingular matrices $\widetilde P_\tau$ and $\widetilde Q_\tau$ such that
\begin{equation}\label{eq:nfp}
\widetilde P_\tau\,(A-\lambda B+\tau \, U\,(D_A-\lambda D_B)\,V^*)\,\widetilde Q_\tau=\left[\begin{array}{ccc}R(\lambda)&0&0\\
0& R_{\rm pre}(\lambda)&0\\ 0&0&R_{\rm ran}(\lambda)\end{array}\right],
\end{equation}
where $R(\lambda)$ is the regular part of the original pencil $A-\lambda B$,
$R_{\rm pre}(\lambda)=D_A-\lambda D_B$, and $R_{\rm ran}(\lambda)$ is regular and independent of $\tau$.
\end{itemize}
For the remaining items, let $\tau\ne 0$ be fixed.
\begin{itemize}
\item[2)] If $\lambda_0$ is an eigenvalue of $R(\lambda)$ and $x$ and $y$ are corresponding right and left eigenvectors
of $A-\lambda B+\tau \, U\,(D_A-\lambda D_B)\,V^*$, then $V^*x=0$ and $U^*y=0$.
\item[3)] If $\lambda_0$ is an eigenvalue of $R_{\rm pre}(\lambda)$ and $x$ and $y$ are corresponding right and left eigenvectors
of $A-\lambda B+\tau \, U\,(D_A-\lambda D_B)\,V^*$, then $V^*x\ne 0$ and $U^*y\ne 0$.
\item[4)] The eigenvalues of $R_{\rm ran}(\lambda)$ are all simple and distinct from the eigenvalues of $R(\lambda)$ and
$R_{\rm pre}(\lambda)$. If $\lambda_0$ is an eigenvalue of $R_{\rm ran}(\lambda)$ and $x$ and $y$ are corresponding right and left
eigenvectors of $A-\lambda B+\tau \, U\,(D_A-\lambda D_B)\,V^*$, then either $V^*x=0$ and $U^*y\ne 0$, or $V^*x\ne 0$ and $U^*y=0$.
\end{itemize}
\end{theorem}

\begin{remark}\rm
We note that in \cite[Thm.~4.3]{HMP19} it was claimed that the result is
true just assuming that the pencil $D_A-\lambda D_B$ is regular rather
than diagonal and regular, but in the proof it was implicitly assumed
that it has the form $\operatorname{diag}(\alpha_1,\dots,\alpha_k)-\lambda\operatorname{diag}(\beta_1,\dots,\beta_k)$.
However, this immediately generalizes to diagonalizable regular pencils
$D_A-\lambda D_B$, because the strict equivalence transformation that
simultaneously diagonalizes the two matrices can be put into $U$ and $V$
using the fact that generic subsets of $\mathbb C^{n,k}$ stay generic under
multiplication with an invertible matrix.
\end{remark}

\begin{remark}\label{rem16.8.19}\rm
It follows from the proof of \cref{thm:nfp} (which is \cite[Thm.~4.6]{HMP19}) that random eigenvalues with corresponding left eigenvectors
satisfying $U^*y= 0$ are generated by a perturbation of a singular block $L_{n_i}^\top$ in the Kronecker canonical form
of $A-\lambda B$ that corresponds to a nonzero left minimal index $n_i$, while random eigenvalues with corresponding right eigenvectors
satisfying $V^*x=0$ come from a perturbation of a singular block $L_{m_j}$ corresponding to a nonzero right minimal index $m_j$.
In particular, the sums $N$ and $M$ of the left minimal indices or right minimal indices, respectively, coincide with the numbers
of random eigenvalues of ``type $U^*y= 0$'' or ``type $V^*x= 0$'', respectively.
\end{remark}

\begin{remark}\label{rem:mrs}\rm As a side result of \cref{thm:nfp} we get from \eqref{pert} also
a basis for the minimal reducing subspace ${\cal M}_{\rm RS}(A,B)$ (see \cite{VD83}).
The basis
is composed of all right eigenvectors $x$ that belong either to a prescribed eigenvalue or to a random eigenvalue such that $U^*y\ne 0$.
In other words, if we collect all right eigenvectors $x$ from eigentriplets $(\lambda,x,y)$ of \eqref{pert} such that $U^*y\ne 0$, these
vectors form a basis for ${\cal M}_{\rm RS}(A,B)$. In a similar way, all left eigenvectors $y$ from eigentriplets $(\lambda,x,y)$ of \eqref{pert} such that $V^*x\ne 0$, form a basis for the left minimal reducing subspace ${\cal L}_{\rm RS}(A,B)$.
\end{remark}

\begin{remark}\rm
It is interesting to compare Theorem~\ref{thm:nfp} to the results from \cite{BolVD94} where it has been shown that adding particular rows
to a rectangular pencil $A_1-\lambda B_1$ with a nontrivial regular part can be used to create additional eigenvalues with values in desired locations while the other eigenvalues of the pencil remain unchanged. In terms of Theorem~\ref{thm:nfp} this corresponds to adding zero rows to the pencil to make it square and then considering a particular rank completing perturbation by replacing the zero rows with nonzero ones. This results in a perturbed pencil of the form
\[
\left[\begin{array}{c}A_1\\ 0\end{array}\right]
-\lambda \left[\begin{array}{c}B_1\\ 0\end{array}\right]
+
\left[\begin{array}{c}0\\ A_2\end{array}\right]
-\lambda \left[\begin{array}{c}0\\ B_2\end{array}\right]
=\left[\begin{array}{c}A_1\\ A_2\end{array}\right]
-\lambda \left[\begin{array}{c}B_1\\ B_2\end{array}\right].
\]
While Theorem~\ref{thm:nfp} indicates that not only for this special type of perturbations, but generically for all rank-completing perturbations the original eigenvalues of the pencil remain unchanged, the results in \cite{BolVD94} focus on how $A_2$ and $B_2$ can be chosen to place the newly generated ``random'' eigenvalues in desired locations.
\end{remark}

We can show that the set of random eigenvalues from \cref{thm:nfp} does not depend on the choice of $\tau$, $D_A$, or $D_B$. Note that the independency of $\tau$ has already been addressed in \cite[Thm.~4.6]{HMP19}.

\begin{lemma}\label{lemUV}
Let $A-\lambda B$ be an $n\times n$ singular pencil of normal rank $n-k$. Then
the random eigenvalues of the perturbed pencil \eqref{pert} depend only
on $U$ and $V$, but not on $\tau$, $D_A$, or $D_B$.
\end{lemma}
\begin{proof}
Suppose that $\mu$ is a random eigenvalue of \eqref{pert} with a
right eigenvector $x$ and a left eigenvector $y$ such that
$V^*x=0$ and $U^*y\ne 0$. Then clearly $(A-\mu B)\,x + \tau\, U(D_A-\mu D_B)V^*x=0$ independently of $\tau$, $D_A$, and $D_B$.
Analogously, when $V^*x\ne 0$ and $U^*y=0$, then
$y^*(A-\mu B) + \tau\, y^*U(D_A-\mu D_B)V^*=0$
independently of $\tau$, $D_A$, and $D_B$.
\end{proof}

In \cref{thm:nfp} we have shown that left and right eigenvectors of true eigenvalues of \eqref{pert} are orthogonal to $U$ and $V$, respectively.
By reviewing the proof of \cite[Thm.~4.3]{HMP19} the result on orthogonality can be extended to root vectors of true eigenvalues as well, that is, for the case when true eigenvalues are multiple and not semisimple.
To show this, we first quote another key result from \cite{HMP19}.

\begin{proposition}\mbox{\rm\cite[Prop.~4.2]{HMP19}}\label{prop4.2}
Let $A-\lambda B$ be an $n\times m$ singular matrix pencil having at least $k$ left minimal indices,
and let $U\in \C^{n,k}$. Then there exists a generic set $\Omega_1\subseteq\C^{n,k}$ such that for
each $U\in\Omega_1$ there exist nonsingular matrices $P$, $Q$ such that
\[
P\,(A-\lambda B)\,Q=\left[\begin{array}{cc}R(\lambda)&0\\ 0& S(\lambda)\end{array}\right]
\quad\mbox{and}\quad
PU=\left[\begin{array}{c}0\\ \widetilde U\end{array}\right],
\]
where $R(\lambda)$ and $S(\lambda)$ are the regular and singular part of $A-\lambda B$,
respectively, and $PU$ is partitioned conformably with $P\,(A-\lambda B)\,Q$.
\end{proposition}

\begin{corollary}\,[Extension of \cite[Thm.~4.3]{HMP19}]\label{cor:root}
Under the same assumptions as is \cref{thm:nfp}, the following assertions
additionally hold for the perturbed pencil \eqref{pert}.
\begin{enumerate}
\item[1)] If $\lambda_0$ is a finite eigenvalue of $R(\lambda)$, then, for $d\ge 1$:
\begin{enumerate}
\item If the vector $x_d$ is a right root vector of height $d$ for $\lambda_0$, i.e.,
there exist nonzero vectors $x_1,\ldots,x_{d-1}$ such that
$(\widetilde A-\lambda_0\widetilde B)x_1=0$ and $(\widetilde A-\lambda_0\widetilde B)x_{i+1}=\widetilde B x_i$ for $i=1,\ldots,d-1$,
then $V^*x_d=0$.
\item If the vector $y_d$ is a left root vector of height $d$ for $\lambda_0$, i.e.,
there exist nonzero vectors $y_1,\ldots,y_{d-1}$ such that
$y_1^*(\widetilde A-\lambda_0\widetilde B)=0$ and $y_{i+1}^*(\widetilde A-\lambda_0\widetilde B)=y_i^*\widetilde B $ for $i=1,\ldots,d-1$,
then $U^*y_d=0$.
\end{enumerate}
\item[2)] If $\infty$ is an eigenvalue of $R(\lambda)$, then, for $d\ge 1$:
\begin{enumerate}
\item If the vector $x_d$ is a right root vector of height $d$ for $\infty$, i.e.,
there exist nonzero vectors $x_1,\ldots,x_{d-1}$ such that
$\widetilde B x_1=0$ and $\widetilde B x_{i+1}=\widetilde A x_i$ for $i=1,\ldots,d-1$, then $V^*x_d=0$.
\item If the vector $y_d$ is a left root vector of height $d$ for $\infty$, i.e.,
there exist nonzero vectors $y_1,\ldots,y_{d-1}$ such that
$y_1^* \widetilde B=0$ and $y_{i+1}^*\widetilde B=y_i^*\widetilde A$ for $i=1,\ldots,d-1$, then $U^*y_d=0$.
\end{enumerate}
\end{enumerate}
\end{corollary}

\begin{proof}

Applying \cref{prop4.2}, there exist nonsingular matrices $P$, $Q$ such that
\begin{equation}\label{eq:help1}
P\,(A-\lambda B)\,Q = \left[\begin{array}{cc}R(\lambda)&0\\ 0&S(\lambda)\end{array}\right],\quad
PU=\left[\begin{array}{c}0\\ U_2\end{array}\right],\quad Q^*V=\left[\begin{array}{c}V_1\\ V_2\end{array}\right],
\end{equation}
where 
$R(\lambda)$ and $S(\lambda)$ are the regular and singular parts of $A-\lambda B$, respectively, and $U$ and $V$ are partitioned conformably with $A-\lambda B$.
The perturbed pencil \eqref{pert} takes the form
\[
P\,(\widetilde A-\lambda \widetilde B)\,Q = \left[\begin{array}{cc}R(\lambda) & 0\\
\tau \, U_2\,(D_A-\lambda D_B)\,V_1^* & R_{\rm new}(\lambda)\end{array}\right],
\]
where
\[
R_{\rm new}(\lambda):=S(\lambda)+\tau \, U_2\,(D_A-\lambda D_B)\,V_2^*.
\]
First we will show 1(b).
We partition $y_i^*P^{-1}=[y_{i1}^*\ \ y_{i2}^*]$ conformably with \eqref{eq:help1} for $i=1,\ldots,d$.
Since $\lambda_0$ is not an eigenvalue of $R_{\rm new}(\lambda)$
we obtain from $y_1^*(\widetilde A-\lambda_0 \widetilde B)=0$ that $y_{12}=0$.
Now we can show that $y_{i2}=0$ implies $y_{i+1,2}=0$ for $i=1,\ldots,d-1$.
If we denote $A_{\rm new}-\lambda B_{\rm new}:=R_{\rm new}(\lambda)$,
 we get from
$y_{i+1}^*(\widetilde A-\lambda_0\widetilde B)= y_i^*\widetilde B$ that
\[
[\times\ \ y_{i+1,2}^*R_{\rm new}(\lambda_0)] = y_{i+1}^*P^{-1}P\,(\widetilde A-\lambda_0\widetilde B)\,Q = y_i^*P^{-1}P\widetilde BQ = [\times\ \ y_{i2}^*B_{\rm new}],
\]
where $\times$ denotes the block that is not important.
It follows from $y_{i2}=0$ and the nonsingularity of $R_{\rm new}(\lambda_0)$ that $y_{i+1,2}=0$.
Consequently, we have $y_i^*U=y_i^*P^{-1}PU=0$ for $i=1, \ldots, d$.

To show 1(a), we apply the already proved part 1(b) to the
pencil $A^*-\lambda B^*$ and the perturbation $V\,(D_A^*-\lambda D_B^*)\,U^*$.
To show 2) we apply the already proved part 1) to the pencil $\lambda A-B$.
\end{proof}

We are now in the position to explore new alternatives to the rank-completing perturbation approach.

\section{Projections of singular pencils}\label{sec:proj}
In this section, we will develop the theory that leads to Algorithm~1 in the next
section. To this end, consider a given singular pencil $A-\lambda B$ and
matrices $U$, $V$, $D_A$, $D_B$ as in \cref{thm:nfp}. Then choosing matrices
$U_\perp$ and $V_\perp$ whose columns form bases for the orthogonal
complement of the ranges of $U$ and $V$, respectively, we can build
an equivalent pencil
\begin{equation}\label{coord}
\widehat A-\lambda \widehat B:= [U \ \, U_{\perp}]^*\,(A - \lambda B)\,[V \ \, V_{\perp}]
= \left[
\begin{array}{cc}
U^*(A-\lambda B)V & U^*(A-\lambda B)V_{\perp} \\[1mm]
U_{\perp}^*(A-\lambda B)V & U_{\perp}^*(A-\lambda B)V_{\perp}
\end{array}
\right].
\end{equation}
A key observation for the development of the new methods
is that with respect to \eqref{coord},
the perturbed pencil $A-\lambda B+\tau \, U\,(D_A-\lambda D_B)\,V^*$ only differs from the original pencil
by an extra term $\tau \, (D_A-\lambda D_B)$ in the $(1,1)$-block.
The following simple result shows that under mild additional assumptions
it is possible to extract the true eigenvalues of the pencil $A-\lambda B$ from
the pencil $U_{\perp}^*(A-\lambda B)V_{\perp}$
of size $(n-k)\times (n-k)$, which ultimately leads to Algorithm~1.

\begin{proposition}\label{prop_proj}
Let $A-\lambda B$ be a complex $n\times n$ singular pencil of normal rank $n-k$
such that all its eigenvalues are semisimple. Furthermore, let $U$, $V$, $D_A$,
$D_B$ satisfy the hypotheses of \cref{thm:nfp} and assume that the
$(n-k)\times (n-k)$ pencil $A_{22}-\lambda B_{22}:=U_{\perp}^*(A-\lambda B)\,V_{\perp}$ from \eqref{coord} is regular. Then,
the eigenvalues of $A_{22}-\lambda B_{22}$ are precisely:
\begin{enumerate}
\item[a)] the random eigenvalues of \eqref{pert} with the same $U$ and $V$;
\item[b)] the true eigenvalues of $A-\lambda B$.
\end{enumerate}
\end{proposition}
\begin{proof}
We know from \cref{summry} and \cref{lemUV} that true eigenvalues of
$A-\lambda B$ and random eigenvalues of \eqref{pert} are independent of
$\tau,D_A$, and $D_B$.
Let $\mu$ be a true or a random eigenvalue of \eqref{pert}.
It follows from \cref{summry} that then
the right eigenvector is $V_{\perp} s$ for
a nonzero $s\in \C^{n-k}$ or the left eigenvector
is $U_{\perp} t$ for a nonzero $t\in \C^{n-k}$ (both statements are true for a true eigenvalue and exactly one of
the statements is true for a random eigenvalue). Then $(A_{22}-\mu B_{22})s=0$ or $t^*(A_{22}-\mu B_{22})=0$, therefore
$\mu$ is an eigenvalue of
$A_{22}-\lambda B_{22}$.
We know from \cref{summry} that \eqref{pert} altogether has $n-k$ eigenvalues of type a) and b).
By a simple counting argument this gives all eigenvalues of $A_{22}-\lambda B_{22}$.
\end{proof}

\cref{prop_proj} suggests to extract the finite eigenvalues of a singular pencil
by projection to a regular pencil of smaller size---an approach that is followed up in Algorithm~1 in the next section.
However, for a rigorous justification of that
algorithm, we have to prove that the projected pencil
$U_{\perp}^*(A-\lambda B)V_{\perp}$ is generically regular.
The challenge here is
that genericity in \cref{thm:nfp} is expressed in terms of the matrices
$U$ and $V$ instead of the matrices $W:=U_\perp$ and $Z:=V_\perp$, which
depend on the choice of $U$ and $V$ although the latter matrices are no longer
needed to form the projected pencil. This challenge is taken care
of in the next result which is the main theoretical result in the
paper and shows that the algorithm will also
work for rectangular pencils and that the assumption of semisimplicity of
eigenvalues in \cref{prop_proj} can be dropped.

\begin{theorem}\label{main}
Let $A-\lambda B$ be a complex $n\times m$ matrix pencil having normal rank $n-k$ and assume $n \ge m$.
Then there exists a generic set $\Omega\subseteq\C^{n,n-k}$ with the property that for each $W\in\Omega$ there
exists a generic set $\Omega'\subseteq\C^{n,n-k}$ such that for all
\begin{equation}\label{eq:11.8.19}
\widehat Z=\left[\begin{array}{c}Z\\ Z'\end{array}\right]\in\Omega',\quad Z\in \C^{m,n-k},\ Z'\in \C^{n-m,n-k}
\end{equation}
the following statements hold:
\begin{enumerate}
\item The $(n-k)\times(n-k)$ pencil $W^*(A-\lambda B)Z$ is regular and has the Kronecker canonical form
\[
\left[\begin{array}{cc}R(\lambda)&0\\ 0&R_{\rm ran}(\lambda)\end{array}\right],
\]
where $R(\lambda)$ coincides with the regular part of $A-\lambda B$ and where all eigenvalues of $R_{\rm ran}(\lambda)$ are simple
and distinct from the eigenvalues of $R(\lambda)$.
\item Let the columns of $W_\perp,\widehat Z_\perp\in \C^{n,k}$ form
bases of the orthogonal complements of the ranges of $W$ and $\widehat Z$, respectively, and let
$Z_\perp\in \C^{m,k}$ denote the upper $m\times k$ submatrix of $\widehat Z_\perp$. Furthermore, let
$\lambda_0$ be an eigenvalue of $W^*(A-\lambda B)Z$ with corresponding left eigenvector $y$ and right eigenvector $x$.
If $\lambda_0\in \C$, then $\lambda_0$ is an eigenvalue of $A-\lambda B$, i.e., of $R(\lambda)$, if and only if
$y^*W^*(A-\lambda_0 B)Z_\perp=0$ and $W_\perp^*(A-\lambda_0 B)Zx= 0$. If $\lambda_0=\infty$ then $\lambda_0$ is an eigenvalue of
$A-\lambda B$, i.e., of $R(\lambda)$, if and only if
$y^*W^*BZ_\perp=0$ and $W_\perp^*BZx= 0$.
\end{enumerate}
\end{theorem}

The proof of Theorem~\ref{main} is rather lengthy and involved and is deferred to Section~\ref{sec:proof}.

\section{Numerical methods for extracting the eigenvalues of singular pencils}

After having established the supporting theory we now present numerical methods that extract
the eigenvalues of a given singular matrix pencil.

\subsection{First extraction method: projection}\label{subsec:proj}
\Cref{main} justifies the following algorithm that returns the regular
eigenvalues $\lambda_i$ of a given singular pencil $A-\lambda B$ 
for $i=1,\ldots,r$, where $r$ is the size of the regular part of
$A-\lambda B$. To further extract finite eigenvalues from all eigenvalues of the
regular part, we will apply Algorithm~3 from \cref{sec:classif}.
The latter algorithm will need as an input the reciprocals of the condition numbers of the regular eigenvalues
of the corresponding projected regular pencil that we therefore
add as an output to the following algorithm. For the theoretical
background on the condition numbers we refer to \cref{sec:classif}.

\noindent\vrule height 0pt depth 0.5pt width \textwidth \\
{\bf Algorithm~1: Computing regular eigenvalues of a singular pencil $A-\lambda B$
by projection}. \\[-3mm]
\vrule height 0pt depth 0.3pt width \textwidth \\
{\bf Input:} $A,B\in \C^{n,m}$, $n \ge m$, $k = n-\text{nrank}(A,B)$,
threshold $\delta$ (default $\varepsilon^{1/2}$).\\
{\bf Output:} Eigenvalues and reciprocals of the condition numbers of the regular eigenvalues. \\
\begin{tabular}{ll}
{\footnotesize 1:} & Select random unitary $n\times n$ matrices 
 $[W \ \, W_\perp]$ and
 $[\widehat Z \ \, \widehat Z_\perp]$,
where $W$ and $\widehat Z$ have $n-k$\\ & columns.\\
{\footnotesize 2:} & Form $Z$ and $Z_\perp$ by removing the last $n-m$ rows of $\widehat Z$ and $\widehat Z_\perp$, respectively.\\
{\footnotesize 3:} & Compute the eigenvalues $\lambda_i$, $i=1,\dots,n-k$, and normalized right and left eigenvectors
$x_i$ and\\ & $y_i$ of $W^*(A-\lambda B)Z$.\\
{\footnotesize 4:} & Compute $\alpha_i=\|W_\perp^*(A-\lambda_iB)Zx_i\|$, $\beta_i=\|y_i^*W^*(A-\lambda_iB)Z_\perp\|$,
$i=1,\dots,n-k$. \\
{\footnotesize 5:} & Compute $\gamma_i=|y_i^*W^*BZx_i|\,(1+|\lambda_i|^2)^{-1/2}$ for $i=1,\ldots, n-k$.\\
{\footnotesize 6:} & Return $\lambda_i$ and $\gamma_i$ for those
$i=1,\ldots,n-k$ such that $\max(\alpha_i,\beta_i) < \delta \, (\|A\|+|\lambda_i|\,\|B\|)$.
\end{tabular} \\
\vrule height 0pt depth 0.5pt width \textwidth
\medskip

One clear advantage of Algorithm~1 over the perturbation method from
\cite{HMP19} is the smaller size of the regular pencil $W^*(A-\lambda B)Z$ which
is $(n-k)\times (n-k)$ instead of $n\times n$. We note that the computation of eigenvalues and left
and right eigenvectors of the $n\times n$ perturbed pencil \eqref{pert}, where
typically the QZ algorithm is applied,
is by far the most expensive step of the method from \cite{HMP19}. Compared to that,
the extra work needed to apply the projections
with matrices $W$ and $Z$ in Algorithm~1, although it can take
$\mathcal{O}(n^3)$ operations, is in practice negligible
and pays off with smaller matrices in Step~{\footnotesize 3} of Algorithm~1,
which is thus usually more efficient even
when $k$ is small compared to $n$.

A strategy that circumvents the need of computing projection matrices is the
following. As the set of matrices $W$, $\widehat Z$ for which the statements in
\cref{main} hold is a generic set, it is likely that it contains all matrices
whose columns are the first $n-k$ columns of a permutation matrix. (Indeed, this
is expected to be the case with probability one for random singular matrix
pencils, but if special pencils from applications are considered, one has
to be careful as the selected matrices may happen to be in the complements of
the generic sets from \cref{main}.)
In that case, the matrices $[W \ \, W_\perp]$
and $[\widehat Z \ \, \widehat Z_\perp]$ can be chosen to be permutation matrices and hence
the pencil $W^*(A-\lambda B)Z$ from \cref{main} can be obtained by extracting $n-k$ rows and columns
from $A-\lambda B$ without any computational cost. Clearly, the columns of $\widehat Z$ should consist of vectors from the
first $m$ standard basis vectors of $\C^n$, otherwise, the corresponding split off matrix $Z$ would contain a zero column.

In the special case $n-k=m<n$ (i.e., the pencil $A-\lambda B$ is of full rank, but nonsquare and thus still singular)
the approach above will render the matrix $Z_\perp=0$, because the matrix $\widehat Z_\perp$ necessarily has to
consist of the last $n-m$ standard basis vectors of $\C^n$. This fits with the observation in \cref{rem16.8.19},
because in this case, the extended pencil $\widetilde A-\lambda\widetilde B$ only has left minimal indices that are zero
and consequently, each random eigenvalue $\lambda_0$ of $W^*(A-\lambda B)Z$ will satisfy the condition
$W_\perp^*(A-\lambda_0 B)Zx=0$, where $x$ is a right eigenvector of $W^*(A-\lambda B)Z$ associated with $\lambda_0$.

\subsection{Second alternative extraction method: augmentation}
\label{sec:augm}
We now present a second new alternative approach to the method of \cite{HMP19} via the $(n+k) \times (n+k)$
augmented (or bordered) matrix pencil
\begin{equation} \label{augm1x}
A_a-\lambda B_a :=
\left[ \begin{array}{cc} A & UT_A \\ S_AV^* & 0 \end{array} \right]
-
\lambda \, \left[ \begin{array}{cc} B & UT_B \\ S_BV^* & 0 \end{array} \right],
\end{equation}
where $S_A,S_B,T_A$, and $T_B$ are $k\times k$ diagonal matrices.
Observe that this approach can be interpreted as applying a rank-completing
perturbation of rank $2k$ to the $(n+k)\times(n+k)$ matrix pencil $\operatorname{diag}(A,0)-\lambda\operatorname{diag}(B,0)$ of normal rank $n-k$
and hence it is expected
that generically the augmented pencil will be regular. However, this does not
immediately follow from the results in \cite{HMP19}, because of the special block
structure of the perturbation pencil.
To simplify the presentation, we will not give a rigorous proof analogue to \cref{main} from the previous section. Instead, we present
the following result, similar to \cref{prop_proj}, that motivates Algorithm~2.

\begin{proposition}\label{prop_augm}
Let $A-\lambda B$ be an $n\times n$ singular pencil of normal
rank $n-k$ such that all its eigenvalues are semisimple, i.e., $d=1$ for all
blocks $J_d$ and $N_d$ in the KCF of $A-\lambda B$. Assume that the regular diagonal
$k\times k$ pencils $S_A-\lambda S_B$ and $T_A-\lambda T_B$ are chosen in such
a way that their $2k$ eigenvalues are pairwise distinct. Furthermore, let
$U,V\in \C^{n,k}$ have orthonormal columns such that
the augmented pencil \eqref{augm1x} is regular. Then the pencil \eqref{augm1x}
has the following eigenvalues:
\begin{enumerate}
\item[a)] $2k$ prescribed eigenvalues, which are precisely the
eigenvalues of $S_A-\lambda S_B$ and $T_A-\lambda T_B$;
\item[b)] the random eigenvalues of \eqref{pert} with the same $U$ and $V$ and with $D_A=T_AS_A$ and $D_B=T_BS_B$;
\item[c)] the true eigenvalues of $A-\lambda B$.
\end{enumerate}
\end{proposition}
\begin{proof}
For part a), clearly, if $\mu$ is an eigenvalue of $T_A-\lambda T_B$ with
an eigenvector $e_i\in \C^k$, then
\[ (A_a-\mu B_a)\left[ \begin{array}{cc} 0 \\ e_i \end{array} \right]=0 \]
and $\mu$ is an eigenvalue of \eqref{augm1x}. In a similar way,
if $\mu$ is an eigenvalue of $S_A-\lambda S_B$ with
an eigenvector $e_i\in \C^k$, then
\[
[\, 0^* \ \ e_i^*\,] \, (A_a-\mu B_a)=0
\]
and $\mu$ is an eigenvalue of \eqref{augm1x}. As we assumed that all eigenvalues of $T_A-\lambda T_B$ and
$S_A-\lambda S_B$ are pairwise distinct, this gives $2k$ eigenvalues of
$A_a-\lambda B_a$.

For the cases b) and c) we consider the perturbed pencil \eqref{pert} with the same $U$ and $V$ and with $D_A=T_AS_A$ and $D_B=T_BS_B$. Then we make use of the fact that by \cref{lemUV} the random eigenvalues of \eqref{pert} are independent of $\tau, D_A$, and $D_B$.
Thus, if $\mu$ is a random eigenvalue of \eqref{pert} then it follows from \cref{summry} that either
the right eigenvector is $V_{\perp} s$ for
a nonzero $s\in \C^{n-k}$ or the left eigenvector
is $U_{\perp} t$ for a nonzero $t\in \C^{n-k}$. Then either
\[
(A_a-\mu B_a)\left[ \begin{array}{cc} V_{\perp} s \\ 0 \end{array} \right]=0\quad{\rm or}\quad
[\, t^*U_{\perp}^* \ \ 0^* \,] \, (A_a-\mu B_a)=0
\]
and consequently $\mu$ is an eigenvalue of $A_a-\lambda B_a$.

For c), similarly as in b), we know from \cref{summry} that $\mu$ is a true eigenvalue of
$A-\lambda B$ when $\mu$ is an eigenvalue
of \eqref{pert} with
the right and left eigenvector
of the form $V_{\perp} s$ and $U_{\perp} t$,
where $s,t\in \C^{n-k}$.
It follows that
\[
(A_a-\mu B_a)\left[ \begin{array}{cc} V_{\perp} s \\ 0 \end{array} \right]=0\quad{\rm and}\quad
 [\, t^*U_{\perp}^* \ \ 0^* \,] \, (A_a-\mu B_a)=0,
\]
therefore $\mu$ is an eigenvalue of
$A_a-\lambda B_a$.

As all eigenvalues in b) are semisimple and all eigenvalues in a) and c) are simple, it follows by a simple counting argument that this gives all eigenvalues of $A_a-\lambda B_{a}$.
\end{proof}

From the above proof we can learn how to extract the
true eigenvalues of $A-\lambda B$ from the eigenvalues
of $A_{a}-\lambda B_{a}$.
Suppose that $(\theta, x, y)$ is an eigentriplet of
$A_{a}-\lambda B_{a}$, where
\[
x = \left[ \begin{array}{cc} x_1 \\ x_2 \end{array} \right],\quad
y = \left[ \begin{array}{cc} y_1 \\ y_2 \end{array} \right] \]
are in block form in accordance with \eqref{augm1x}. It is
easy to see that $\theta$ is a true eigenvalue if and only if
$x_2=y_2=0$. Here we assume that the diagonal matrices
$T_A,T_B,S_A,S_B$ are chosen in such a way that all prescribed eigenvalues differ from the true eigenvalues of $A-\lambda B$.
We give an overview of the method, based
on the augmented pencil \eqref{augm1x}, in Algorithm~2. In practice, this algorithm can be applied to an arbitrary singular pencil, not just to those covered
by \cref{prop_augm}. As for Algorithm~1, to further extract finite eigenvalues from all eigenvalues of the
regular part, we apply Algorithm~3 from \cref{sec:classif}.

\noindent\vrule height 0pt depth 0.5pt width \textwidth \\
{\bf Algorithm~2: Computing regular eigenvalues of a singular pencil $A-\lambda B$
by an augmentation}. \\[-3mm]
\vrule height 0pt depth 0.3pt width \textwidth \\
{\bf Input:} $A,B\in \C^{n,n}$, $k = n-\text{nrank}(A,B)$,
threshold $\delta$ (default $\varepsilon^{1/2}$).\\
{\bf Output:} Eigenvalues and reciprocals of the condition numbers of the regular eigenvalues. \\
\begin{tabular}{ll}
{\footnotesize 1:} & Select random $n\times k$ matrices $U$ and $V$ with orthonormal columns.\\
{\footnotesize 2:} & Select diagonal $k \times k$ matrices $T_A$, $T_B$, $S_A$, and $S_B$ such that the
eigenvalues of $T_A-\lambda T_B$ and\\
& $S_A-\lambda S_B$ are (likely) different from those of $A-\lambda B$ (default: choose diagonal elements\\
& uniformly random from the interval $[1, 2]$).\\
{\footnotesize 3:} & Compute the eigenvalues $\lambda_i$, $i=1,\dots,n+k$, and normalized right and left eigenvectors
$\smtxa{c}{x_{i1} \\ x_{i2}}$ \\
& and $\smtxa{c}{y_{i1} \\ y_{i2}}$ of the augmented pencil
\eqref{augm1x}.
\\
{\footnotesize 4:} & Compute $\alpha_i=\|x_{i2}\|$, \ $\beta_i=\|y_{i2}\|$,
\ $i=1,\dots,n+k$. \\
{\footnotesize 5:} & Compute $\gamma_i=|y_{i1}^*Bx_{i1}| \, (1+|\lambda_i|^2)^{-1/2}$, \ $i=1,\ldots, n+k$.\\
{\footnotesize 6:} & Return $\lambda_i$ and $\gamma_i$ for those
$i=1,\ldots,n+k$ such that $\max(\alpha_i,\beta_i)<\delta$.
\end{tabular} \\
\vrule height 0pt depth 0.5pt width \textwidth
\medskip

Compared to the approach of \eqref{pert} and Algorithm~1,
in Algorithm~2 we have to compute eigenvalues and eigenvectors of a regular pencil of the larger size $(n+k)\times (n+k)$.
Because of that the other two methods, in particular the rank projection method, are more suitable for dense singular pencils.
An advantage of the augmented pencil approach is that it does not change the original matrices $A$ and $B$. If
the matrices are sparse and such that a linear system with
the augmented matrix $A_a-\sigma B_a$ can be solved efficiently, then
we can apply a shift-and-invert subspace method, for instance {\tt eigs} in
Matlab, to compute a subset of true eigenvalues close to a target $\sigma$.

\begin{remark}\rm\label{rem:augmx}
If one is interested only in the finite eigenvalues of the given singular
pencil, then instead of \eqref{augm1x} we can consider a simpler augmented pencil
\begin{equation}
\label{augm1}
\breve{A}-\lambda \breve{B} :=
\left[ \begin{array}{cc} A & U \\ V^* & 0 \end{array} \right]
\left[ \begin{array}{cc} x \\ z \end{array} \right]
=
\lambda \, \left[ \begin{array}{cc} B & 0 \\ 0 & 0 \end{array} \right]
\left[ \begin{array}{cc} x \\ z \end{array} \right].
\end{equation}
This is equivalent to selecting $T_A=S_A=I_k$
and $T_B=S_B=0$ in line 2 of Algorithm~2. Although this case is not
covered by \cref{prop_augm}, numerical tests suggest that the method
still works well.
In this case all $2k$ prescribed eigenvalues are $\infty$, which makes it harder to separate infinite prescribed eigenvalues from
the infinite true eigenvalues, but since
we are only interested in finite true eigenvalues, this is not important. On the other hand, it is guaranteed that the prescribed
eigenvalues are different from the finite eigenvalues of the
given singular pencil.
\end{remark}

\section{Classification of regular eigenvalues into finite and infinite}
\label{sec:classif} Suppose that
using Algorithm~1, Algorithm~2, or the algorithm from \cite{HMP19} we have extracted
the true eigenvalues of the singular pencil.
In the second phase, which is common to all three methods, we extract the
finite eigenvalues from the set of identified true eigenvalues.
In \cite{HMP19}, we have used the values $s_i=y_{i}^*Bx_i$ as a criterion,
where $y_i$ and $x_i$ are normalized left and
right eigenvectors, respectively, corresponding to an eigenvalue $\lambda_i$ of
the perturbed pencil \eqref{pert}. If $\lambda_i$ is a simple
eigenvalue, then $1/|s_i|$ appears in
the expression for a standard condition number. However, this is a condition
number of $\lambda_i$ as an eigenvalue of the regular pencil \eqref{pert},
and for different matrices $U$, $V$ we get different eigenvectors and thus
different values of $s_i$ for the same eigenvalue $\lambda_i$ of the
original singular pencil.

Since an eigenvalue of a singular pencil can always
be perturbed into an arbitrary value using arbitrarily small perturbations it follows that
the condition number of that eigenvalue for the original singular matrix
pencil is infinite. Therefore, Lotz and Noferini suggest in \cite{LotzNoferini} to consider the
so-called $\delta$-weak condition number of eigenvalues for singular pencils
and show that it can be approximated by an expression that is given in the following definition.

\begin{definition}\label{df:kappa}
Let $\lambda_0$ be an algebraically simple eigenvalue of an $n\times n$ singular pencil $A-\lambda B$ that has normal rank $n-k$. Let
$X=[X_1\ \, x]$ be an $n\times (k+1)$ matrix with orthonormal columns such that the columns of $X_1$ form a basis for
${\rm ker}(A-\lambda_0 B)\cap {\cal M}_{\rm RS}(A,B)$ and the columns of $X$ form a basis for
${\rm ker}(A-\lambda_0 B)$, and let
$Y=[Y_1\ y]$ be an $n\times (k+1)$ matrix with orthonormal columns such that the columns of $Y_1$ form a basis for
${\rm ker}((A-\lambda_0 B)^*)\cap {\cal L}_{\rm RS}(A,B)$
and the columns of $Y$ form a basis for
${\rm ker}((A-\lambda_0 B)^*)$. Then we define
\begin{equation}\label{def:kp}
\gamma(\lambda_0)=|y^*Bx| \, (1+|\lambda_0|^2)^{-1/2}
\end{equation}
and use $\kappa(\lambda_0)=\gamma(\lambda_0)^{-1}$ as the \emph{condition number} of
$\lambda_0$.
\end{definition}

Based on the above definition and also on the fact that a similar approach,
supported by strong theoretical results, has recently been used in an algorithm for computing finite eigenvalues of a singular pencil in \cite{KreGlib22}, where full-rank random perturbations are exploited, we opt to use
\begin{equation}\label{def:kpi}
\gamma_i=|y_i^*Bx_i| \, (1+|\lambda_i|^2)^{-1/2}
\end{equation}
in criteria to extract finite eigenvalues.
Here $\lambda_i$ is a computed eigenvalue and $x_i$ and $y_i$ are the corresponding computed normalized right and left eigenvectors, respectively.

Let $X=[X_1\ \, x]$ and $Y=[Y_1\ \, y]$ be bases for subspaces associated with
an eigenvalue $\lambda_i$ as in \cref{df:kappa}. If $U$ and $V$ are the
$n\times k$ matrices used in \eqref{pert}, then the right and left
eigenvectors $x_i$ and $y_i$ have the forms
$x_i=[X_1\ \, x] \smtxa{c}{a \\ \alpha}$ and
$y_i=[Y_1\ \, y] \smtxa{c}{b \\ \beta}$,
and satisfy $\|x_i\|=\|y_i\|=1$, $V^*x_i=0$ and $U^*y_i=0$. Since
$Y^*BX_1=0$ and $Y_1^*BX=0$, we get
$y_i^*Bx_i = \alpha \beta^* y^*Bx$ and thus $\gamma_i=|\alpha| \, |\beta| \, \gamma(\lambda_i)$. Since
$0\le |\alpha|,|\beta|\le 1$, we always get $\gamma_i\le \gamma(\lambda_i)$.
Numerical results suggest that if the elements of $U$ and $V$ are
independent and identically distributed standard normal random variables, then
the expected value of $|\alpha|\,|\beta|$ behaves as
${\cal O}(1/k)$ and the computed condition numbers can thus be used as reliable
criteria to extract simple finite eigenvalues satisfying $\gamma_i\ne 0$. A
detailed stochastic analysis is outside the scope of this paper.

In \cite{HMP19}, we have used the criterion $|y_i^*Bx_i|\ge \delta_2$ with a default
value of $\delta_2=10^2\,\varepsilon$, where $\varepsilon$ is the machine
precision ($\varepsilon = 2.2\cdot 10^{-16}$ in double precision), to detect
finite eigenvalues. The criterion is based on the fact that $y_i^*Bx_i\ne 0$ if
$\lambda_i$ is a simple eigenvalue. However, singular pencils can also have multiple
eigenvalues and for those values
$y_i^*Bx_i$ might be $0$. In practice, an eigenvalue of algebraic multiplicity
$m$ is numerically evaluated
as $m$ simple eigenvalues with nonzero values $y_i^*Bx_i$ and because of that
\cite[Algorithm~1]{HMP19} properly extracts
multiple finite eigenvalues in most cases.

To further improve the detection of finite eigenvalues, we propose the
heuristic criteria presented in Algorithm~3. As another distinctive value for
the recognition we use the relative gap of an eigenvalue, which we define as
\[
{\rm gap}_i=\min_{j\ne i}\frac{|\lambda_j-\lambda_i|}{(1+|\lambda_i|^2)^{1/2}}.
\]
The main idea is that a computed representative $\lambda_i$ of a multiple finite
eigenvalue will have
a small $\gamma_i$ but also a small
${\rm gap}_i$. On the other hand, a (multiple) infinite eigenvalue will usually
appear as a finite eigenvalue with
$\gamma_i$ very close to zero and a large ${\rm gap}_i$ (close to $1$).

\noindent\vrule height 0pt depth 0.5pt width \textwidth \\
{\bf Algorithm~3: Extraction of finite eigenvalues from the regular eigenvalues
of a singular pencil $A-\lambda B$}. \\[-3mm]
\vrule height 0pt depth 0.3pt width \textwidth \\
{\bf Input:} eigenvalues $\lambda_i$
 and reciprocals $\gamma_i$ of condition numbers, $i=1,\ldots,r$, for regular eigenvalues of $A-\lambda B$,
thresholds $\delta_1$ (default $\varepsilon^{1/2}$), $\delta_2$ (default $10^2\,\varepsilon$),
$\xi_1$ (default $0.95$) and $\xi_2$ (default 0.01).\\
{\bf Output:} Finite eigenvalues of the regular part. \\
\begin{tabular}{ll}
{\footnotesize 1:} & Compute ${\rm gap}_{i}=\min_{j\ne i} |\lambda_j-\lambda_i| \, (1+|\lambda_i|^2)^{-1/2}$, \ $i=1,\ldots,r$.\\
{\footnotesize 2:} & If $\gamma_i<\delta_1$ and ${\rm gap}_i>\xi_1$, \ $i=1,\ldots,r$, flag $\lambda_i$ as an infinite eigenvalue.\\
{\footnotesize 3:} & If $\gamma_i<\delta_2$ and ${\rm gap}_i>\xi_2$, \ $i=1,\ldots,r$, flag $\lambda_i$ as an infinite eigenvalue.\\
{\footnotesize 4:} & Return $\lambda_i$ for those $i=1,\ldots,r$ such that $\lambda_i$ is not flagged as an infinite eigenvalue.\\
\end{tabular} \\
\vrule height 0pt depth 0.5pt width \textwidth
\medskip

A brief explanation of the criteria in Algorithm~3 is as follows. If the pencil
$A-\lambda B$ has an eigenvalue $\lambda_0$ with
a Jordan block $J_d(\lambda_0)$, where $d\ge 2$, then a backward stable
numerical algorithm that we apply to the pencil $\widetilde A-\lambda \widetilde B$ from \eqref{pert} will typically compute
$d$ simple eigenvalues $\lambda_0^{(1)},\ldots,\lambda_0^{(d)}$ such that
$|\lambda_0^{(i)}-\lambda_0|={\cal O}(\varepsilon^{1/d})$
for $i=1,\ldots,d$. Therefore, we have ${\rm gap}_i={\cal O}(\varepsilon^{1/d})$.
(The same analysis applies to the pencil $W^*(A-\lambda B)Z$ from Algorithm~1 or
to the augmented pencil from Algorithm~2). As the algorithm is backward stable,
these are exact eigenvalues of a perturbed pencil
$\dot A-\lambda \dot B$, such that $\|\dot A -\widetilde A\|={\cal O}(\varepsilon)\,\|\widetilde A\|$ and
$\|\dot B -\widetilde B\|={\cal O}(\varepsilon)\,\|\widetilde B\|$. By reversing roles of the pencils,
we can consider that a simple eigenvalue $\lambda_0^{(i)}$ of $\dot A-\lambda \dot B$ perturbs into an
eigenvalue $\lambda_0$ of $\widetilde A -\lambda \widetilde B$.
As we know that a change of a simple eigenvalue, which is ${\cal O}(\varepsilon^{1/d})$, is bounded by the product of the size of the perturbation
of the matrices, which is ${\cal O}(\varepsilon)$, and the condition number $\kappa(\lambda_0^{(i)})$ of the eigenvalue, this gives us a lower bound
on the condition number.
It follows that we can expect $\kappa(\lambda_0^{(i)})={\cal O}(\varepsilon^{(1-d)/d})$ and thus $\gamma(\lambda_0^{(i)})={\cal O}(\varepsilon^{(d-1)/d})$. This explains the observation that
in practice we get a small nonzero value $\gamma_i$ in case of a multiple
eigenvalue that is usually large enough so that we recognize the eigenvalue as a
finite one. If however, we find an eigenvalue with a value $\gamma_i$ very close
to zero and a small value of ${\rm gap}_i$, then we can deduce that this also
represents a multiple eigenvalue.

Rather than setting inclusive criteria for finite eigenvalues, we set two
sufficient criteria for infinite eigenvalues and treat the remaining eigenvalues
as finite. In line 2, if we have an eigenvalue with a very large ${\rm gap}_i$
and sufficiently small $\gamma_i$ (we use $\delta_1=\varepsilon^{1/2}$ since we
expect that a simple finite eigenvalue will have $\gamma_i>\delta_1$), then this
indicates an infinite eigenvalue. In line 3, if $\gamma_i$ is very small, this
indicates either a finite eigenvalue of large multiplicity or an infinite
eigenvalue. A finite eigenvalue will also have small ${\rm gap}_i$, on the
contrary to an infinite eigenvalue.
Algorithm~3 usually works well, unless the pencil $A-\lambda B$ has blocks
$J_d$ and $N_d$ in the regular part with a very large size $d$.

\section{Numerical examples}\label{sec:num}
In this section we demonstrate the methods with several numerical examples computed in Matlab 2021b.
All numerical examples and implementations of the algorithms are available in \cite{MultiParEig}.

\begin{example}\rm For the first example we revisit \cite[Ex.~6.1]{HMP19}, where
\begin{equation}\label{eq:AB7x7}
A=\smath{\left[\begin{array}{rrrrrrr}
 -1 & -1 & -1 & -1 & -1 & -1 & -1 \\
 1 & 0 & 0 & 0 & 0 & 0 & 0 \\
 1 & 2 & 1 & 1 & 1 & 1 & 1 \\
 1 & 2 & 3 & 3 & 3 & 3 & 3 \\
 1 & 2 & 3 & 2 & 2 & 2 & 2 \\
 1 & 2 & 3 & 4 & 3 & 3 & 3 \\
 1 & 2 & 3 & 4 & 5 & 5 & 4\end{array}\right]},\quad
B=\smath{\left[\begin{array}{rrrrrrr}
 -2 & -2 & -2 & -2 & -2 & -2 & -2 \\
 2 & -1 & -1 & -1 & -1 & -1 & -1 \\
 2 & 5 & 5 & 5 & 5 & 5 & 5 \\
 2 & 5 & 5 & 4 & 4 & 4 & 4 \\
 2 & 5 & 5 & 6 & 5 & 5 & 5 \\
 2 & 5 & 5 & 6 & 7 & 7 & 7 \\
 2 & 5 & 5 & 6 & 7 & 6 & 6\end{array}\right]},
\end{equation}
${\rm nrank}(A,B)=6$, and the KCF has blocks $J_1(1/2)$, $J_1(1/3)$, $N_1$, $L_1$, and $L_2^\top$.
Algorithm~1 projects $A-\lambda B$ to a $6\times 6$ pencil and returns
the values in \cref{tab:ex71proj}.

\begin{table}[htb!] \label{tab:ex71proj}
\centering
\caption{Results of Algorithm~1 (projection to normal rank) followed by Algorithm~3 applied to the singular pencil \eqref{eq:AB7x7}.}
{\footnotesize \begin{tabular}{c|clllll} \hline \rule{0pt}{2.3ex}%
$j$ & $\lambda_j$ & $\quad \ \gamma_j$ & $\quad \ \alpha_i$ & $\quad \ \beta_j$ & ${\rm gap}_j$ & Type \\[0.5mm]
\hline \rule{0pt}{2.5ex}%
1 & $0.5000000$ & $3.0\cdot 10^{-2}$ & $5.5\cdot 10^{-17}$ & $1.7\cdot 10^{-17}$ & $0.15$ & Finite true \\
2 & $0.3333333$ & $6.1\cdot 10^{-2}$ & $1.6\cdot 10^{-17}$ & $2.9\cdot 10^{-17}$ & $0.16$ & Finite true \\
3 & $-\infty$ & $0.0$ & $8.3\cdot 10^{-17}$ & $1.2\cdot 10^{-17}$ & $1.00$ & Infinite true \\
4 & $1.911415$ & $5.0\cdot 10^{-4}$ & $8.4\cdot 10^{-18}$ & $8.0\cdot 10^{-3}$ & $0.65$ & Random right \\
5--6 & $-0.1929034 \pm 0.3581876i$ & $4.0\cdot 10^{-3}$ & $4.4\cdot 10^{-3}$ & $2.5\cdot 10^{-17}$ & $0.59$ & Random left \\
\hline
\end{tabular}}
\end{table}

\noindent Algorithm~2 uses augmentation, therefore it works with matrices
$8\times 8$ and returns the values in \cref{tab:ex71augmx}.
Notice that the prescribed eigenvalues from pencils $T_A-\lambda T_B$ and $S_A-\lambda S_B$ can be identified from $\alpha_j=1$ and
$\beta_j=1$, respectively. If we use Algorithm~2 with a simpler augmented pencil \eqref{augm1} from \cref{rem:augmx}, we
obtain the values in \cref{tab:ex71augm}. In this case both prescribed eigenvalues are infinite.
There is no noticeable difference in the accuracy
of the computed eigenvalues, Algorithm~1 and Algorithm~2 (using \eqref{augm1x} or \eqref{augm1}) both return eigenvalues such that the
distance to the exact ones is ${\cal O}(10^{-16})$.

\begin{table}[htb!] \label{tab:ex71augmx}
\centering
\caption{Results of Algorithm~2 (augmented pencil) followed by Algorithm~3 applied to the singular pencil \eqref{eq:AB7x7}.}
{\footnotesize \begin{tabular}{c|clllll} \hline \rule{0pt}{2.3ex}%
$j$ & $\lambda_j$ & $\quad \ \gamma_j$ & $\quad \ \alpha_i$ & $\quad \ \beta_j$ & ${\rm gap}_j$ & Type \\[0.5mm]
\hline \rule{0pt}{2.3ex}%
1 & $0.5000000$ & $3.0\cdot 10^{-2}$ & $2.4\cdot 10^{-16}$ & $6.1\cdot 10^{-16}$ & $0.10$ & Finite true \\
2 & $0.3333333$ & $2.2\cdot 10^{-3}$ & $1.8\cdot 10^{-16}$ & $2.4\cdot 10^{-16}$ & $0.05$ & Finite true \\
3 & $\infty$ & $0.0$ & $9.5\cdot 10^{-16}$ & $1.5\cdot 10^{-16}$ & $1.00$ & Infinite true \\
4 & $0.9449265$ & $1.9\cdot 10^{-1}$ & $1.5\cdot 10^{-15}$ & $1.0$ & $0.23$ & Prescribed \\
5 & $0.3863457$ & $9.5\cdot 10^{-4}$ & $4.3\cdot 10^{-16}$ & $2.6\cdot 10^{-3}$ & $0.50$ & Random right \\
6 & $0.6307942$ & $6.9\cdot 10^{-2}$ & $1.0$ & $3.8\cdot 10^{-16}$ & $0.11$ & Prescribed \\
7--8 & $-0.3451090 \pm 0.3590986i$ & $3.8\cdot 10^{-3}$ & $7.4\cdot 10^{-3}$ & $4.6\cdot 10^{-17}$ & $0.64$ & Random left \\
\hline
\end{tabular}}
\end{table}

\begin{table}[htb!] \label{tab:ex71augm}
\centering
\caption{Results of Algorithm~2 (augmented pencil) using pencil \eqref{augm1} followed by Algorithm~3 applied to the singular pencil \eqref{eq:AB7x7}.}
{\footnotesize \begin{tabular}{c|clllll} \hline \rule{0pt}{2.3ex}%
$j$ & $\lambda_j$ & $\quad \ \gamma_j$ & $\quad \ \alpha_i$ & $\quad \ \beta_j$ & ${\rm gap}_j$ & Type \\[0.5mm]
\hline \rule{0pt}{2.5ex}%
1 & $0.5000000$ & $7.6\cdot 10^{-2}$ & $6.3\cdot 10^{-18}$ & $6.5\cdot 10^{-17}$ & $0.15$ & Finite true \\
2 & $0.3333333$ & $7.8\cdot 10^{-2}$ & $9.0\cdot 10^{-17}$ & $4.6\cdot 10^{-17}$ & $0.16$ & Finite true \\
3 & $-0.1817670$ & $1.8\cdot 10^{-3}$ & $3.1\cdot 10^{-17}$ & $1.0\cdot 10^{-2}$ & $0.40$ & Random right \\
4 & $\infty$ & $0.0$ & $6.9\cdot 10^{-17}$ & $1.0$ & $1.00$ & Prescribed \\
5--6 & $-0.2627866 \pm 0.3972325i$ & $2.7\cdot 10^{-3}$ & $1.3\cdot 10^{-3}$ & $5.0\cdot 10^{-17}$ & $0.37$ & Random left \\
7 & $\infty$ & $0.0$ & $1.0$ & $7.2\cdot 10^{-17}$ & $1.00$ & Prescribed \\
8 & $\infty$ & $0.0$ & $1.5\cdot 10^{-2}$ & $5.1\cdot 10^{-17}$ & $1.00$ & Random left \\
\hline
\end{tabular}}
\end{table}

Since in \cref{tab:ex71augm} we use the pencil \eqref{augm1} in Algorithm~2, all prescribed eigenvalues
are infinite. A side effect is that the algorithm fails to identify the true infinite
eigenvalue from block $N_1$. Since we are interested in finite eigenvalues only, this is not very relevant.
\end{example}

\begin{example}\label{ex:gaps}\rm To illustrate the behavior of values
$\gamma_i$ and ${\rm gap}_i$ in Algorithm~3, we give a numerical example with a
singular pencil that has multiple finite eigenvalue with blocks $J_d$ of
different sizes. We construct the pencil in Matlab using the MCS Toolbox \cite{MCS}
as
\begin{verbatim}
    rng('default')
    [A0,B0] = kcf(pstruct([2 1], [2 1], {[4,2,1]}, [1], [2,1]));
    Q = randn(18); Z = randn(18); A = Q*A0*Z; B = Q*B0*Z;
\end{verbatim}
This constructs a pencil $A-\lambda B$ of size $18\times 18$ with the following
blocks in the KCF: $L_1$, $L_2$, $L_1^\top$, $L_2^\top$, $J_4(1)$, $J_2(1)$, $J_1(1)$,
$N_2$, $N_1$. The normal rank is 16 and $\lambda=1$ is an eigenvalue of multiplicity
$7$. If we apply Algorithm~1 followed by Algorithm~3, we get the results
in \cref{tab:ex62}.

\begin{table}[htb!] \label{tab:ex62}
\centering
\caption{Results of Algorithm~1 followed by Algorithm~3 applied to the singular
pencil from \cref{ex:gaps}.
}
{\footnotesize \begin{tabular}{c|clllll} \hline \rule{0pt}{2.3ex}%
$j$ & $\lambda_j$ & $\quad \ \gamma_j$ & $\quad \ \alpha_j$ & $\quad \ \beta_j$ & ${\rm gap}_j$ & Type \\[0.5mm]
\hline \rule{0pt}{2.5ex}%
1 & $1.000000$ & $1.5\cdot 10^{-3}$ & $4.2\cdot 10^{-17}$ & $5.4\cdot 10^{-17}$ & $5.4\cdot 10^{-8}$ & Finite true \\
2--3 & $1.000000 \pm 7.563780\cdot 10^{-8}i$ & $1.5\cdot 10^{-10}$ & $5.4\cdot 10^{-17}$ & $9.1\cdot 10^{-17}$ & $5.4\cdot 10^{-8}$ & Finite true \\
4 & $0.9998633$ & $1.8\cdot 10^{-13}$ & $4.9\cdot 10^{-17}$ & $9.3\cdot 10^{-17}$ & $9.7\cdot 10^{-5}$ & Finite true \\
5--6 & $1.000000 \pm 1.366338\cdot 10^{-4}i$ & $1.8\cdot 10^{-13}$ & $4.2\cdot 10^{-17}$ & $1.0\cdot 10^{-16}$ & $9.7\cdot 10^{-5}$ & Finite true \\
7 & $1.000137$ & $1.7\cdot 10^{-13}$ & $4.4\cdot 10^{-17}$ & $9.5\cdot 10^{-17}$ & $9.7\cdot 10^{-5}$ & Finite true \\
8--9 & $-26.12267 \pm 2.530975\cdot 10^7i$ & $1.7\cdot 10^{-24}$ & $8.2\cdot 10^{-17}$ & $1.4\cdot 10^{-16}$ & $1.00$ & Infinite true \\
10 & $\infty$ & $0.0$ & $6.5\cdot 10^{-17}$ & $1.3\cdot 10^{-16}$ & $1.00$ & Infinite true \\
11 & $-0.3215088$ & $1.8\cdot 10^{-3}$ & $6.1\cdot 10^{-17}$ & $5.3\cdot 10^{-3}$ & $0.45$ & Random right\\
12 & $-0.7977905$ & $1.3\cdot 10^{-3}$ & $6.9\cdot 10^{-16}$ & $5.7\cdot 10^{-4}$ & $0.37$ & Random right\\
13 & $3.177487$ & $3.7\cdot 10^{-4}$ & $4.6\cdot 10^{-17}$ & $3.8\cdot 10^{-3}$ & $0.65$ & Random right\\
14--15 & $-0.1537092 \pm 0.8144710$ & $2.7\cdot 10^{-3}$ & $3.1\cdot 10^{-3}$ & $8.5\cdot 10^{-17}$ & $0.31$ & Random left\\
16 & $-1.672076$ & $4.3\cdot 10^{-4}$ & $9.9\cdot 10^{-4}$ & $7.9\cdot 10^{-17}$ & $0.45$ & Random left\\
\hline
\end{tabular}}
\end{table}

In \cref{tab:ex62} we see a clear gap between regular and random eigenvalues,
\[
\max_{j=1,\ldots,10}(\max(\alpha_j,\beta_j))=1.4\cdot 10^{-16} \ll \min_{j=11,\ldots,16}(\max(\alpha_j,\beta_j))=5.7\cdot 10^{-4},
\]
and Algorithm~1 correctly extracts 10 true eigenvalues.
Eigenvalue $1$ of multiplicity $7$ is computed as seven simple eigenvalues $\lambda_1,\ldots,\lambda_7$. The first one is related to
the Jordan block $J_1(1)$ and $|\lambda_1-1|=5.4\cdot 10^{-15}$. The
eigenvalues $\lambda_2,\lambda_3$ come from the block $J_2(1)$ and
$|\lambda_2-1|=|\lambda_3-1|=7.6\cdot 10^{-8}$, while
$\lambda_4,\lambda_5,\lambda_6,\lambda_7$ originate from the block $J_4(1)$
and $|\lambda_j-1|=1.4\cdot 10^{-4}$ for $j=4,5,6,7$. This is exactly the
behavior that we have predicted in the discussion after Algorithm~3, and this
occurs also in the magnitudes of the gaps of the eigenvalues
$\lambda_2,\ldots,\lambda_7$. The values $\gamma_j$ also follow the predicted
pattern: the products $\gamma_j\cdot {\rm gap}_j$ for
$j=2,\ldots,7$ are all of magnitude $10^{-17}$.
\end{example}

\begin{example}\rm Often, even if the matrix pencil $A-\lambda B$ is singular,
the routine {\tt eig(A,B)} implemented in Matlab still accurately computes true
eigenvalues. In this example we present a matrix pencil, for which
{\tt eig(A,B)} fails completely, while Algorithm~1 or Algorithm~2 both successfully extract
the finite eigenvalues.
The pencil \cite[Ex.~6.2]{HMP19} has the form
\begin{equation}\label{eq:AB5x5}A-\lambda B =
\smath{\left[\begin{array}{rrrrr}
 1 & -2 & 100 & 0 & 0 \\
 1 & 0 & -1 & 0 & 0 \\
 0 & 0 & 0 & 1 & -75 \\
 0 & 0 & 0 & 0 & 2 \\
 0 & 0 & 0 & 0 & 0 \end{array}\right]}
-\lambda \, \smath{\left[\begin{array}{rrrrr}
 0 & 1 & 0 & 0 & 0 \\
 0 & 0 & 1 & 0 & 0 \\
 0 & 0 & 0 & 1 & 0 \\
 0 & 0 & 0 & 0 & 1 \\
 0 & 0 & 0 & 0 & 0 \end{array}\right]}
\end{equation}
and its KCF is composed of the blocks $L_0^\top$, $L_2$, $J_1(1)$, and $J_1(2)$.
For this pencil {\tt eig(A,B)} returns eigenvalues $-2$, $\infty$, {\tt NaN},
{\tt NaN}, {\tt NaN} and none of them is close to any of the correct eigenvalues
$1$ and $2$. On the other hand, Algorithm~1 returns the values in
\cref{tab:ex72augm} and both eigenvalues $1$ and $2$ are correctly identified.

\begin{table}[htb!] \label{tab:ex72augm}
\centering
\caption{Results of Algorithm~1 (projected pencil) followed by Algorithm~3 applied to the singular pencil \eqref{eq:AB5x5}.}
{\footnotesize \begin{tabular}{c|clllll} \hline \rule{0pt}{2.3ex}%
$j$ & $\lambda_j$ & $\quad \ \gamma_j$ & $\quad \ \alpha_i$ & $\quad \ \beta_j$ & ${\rm gap}_j$ & Type \\[0.5mm]
\hline \rule{0pt}{2.5ex}%
1 & $1.000000$ & $4.8\cdot 10^{-3}$ & $3.1\cdot 10^{-16}$ & $1.4\cdot 10^{-17}$ & $0.71$ & Finite true \\
2 & $2.000000$ & $5.0\cdot 10^{-3}$ & $3.4\cdot 10^{-16}$ & $6.8\cdot 10^{-18}$ & $0.45$ & Finite true \\
3--4 & $-0.1672036 \pm 16.44603 i$ & $5.5\cdot 10^{-2}$ & $7.3\cdot 10^{-17}$ & $2.8\cdot 10^{-2}$ & $1.01$ & Random right \\
\hline
\end{tabular}}
\end{table}

To be fair, let us comment that if we multiply $A$ and $B$ by random orthogonal $5\times 5$ matrices $Q$ and $Z$, then
{\tt eig(Q*A*Z,Q*B*Z)} returns $5$ finite eigenvalues that include eigenvalues $\mu_1$ and $\mu_2$ such that
$|\mu_1-1|=4.1\cdot 10^{-12}$
and $|\mu_2-2|=2.6\cdot 10^{-12}$.

In \cite{KreGlib22} a method is proposed that perturbs a singular matrix pencil
$A-\lambda B$ into $(A+\varepsilon E_1) -\lambda \, (B+\varepsilon E_2)$, where $\varepsilon$ is small (default value $\varepsilon=10^{-8}$) and $E_1$, $E_2$ are {\em full-rank} random matrices such that $\|E_1\|=\|E_2\|=1$.
Finite eigenvalues are selected based on the condition numbers of the computed eigentriplets $(\lambda_i,x_i,y_i)$. An advantage of the method from
\cite{KreGlib22} is that it does not require the normal rank.
A disadvantage is that the true eigenvalues are perturbed, and since the
perturbation $\varepsilon$ is not so small (selecting it too small makes it harder to identify the true eigenvalues), the eigenvalues
are computed less accurately compared to Algorithm~1, which theoretically does not move the true eigenvalues.
Indeed, for this example, the method from \cite{KreGlib22} computes the finite eigenvalues
$\widetilde \lambda_1$ and $\widetilde \lambda_2$ such that
$|\widetilde\lambda_1-1|=3.3\cdot 10^{-7}$ and $|\widetilde\lambda_2-2|=4.7\cdot 10^{-7}$. The eigenvalues $\lambda_1$ and $\lambda_2$ computed by Algorithm~1 in \cref{tab:ex72augm} are considerably more accurate:
$|\lambda_1-1|=6.9\cdot 10^{-13}$ and $|\lambda_2-2|=7.6\cdot 10^{-13}$.
\end{example}

\begin{example}\label{sec:poly}\rm
We now consider in more detail a specific source of singular generalized
eigenvalue problems arising from polynomial equations.
Let $p_1(\lambda,\mu)=0$, $p_2(\lambda,\mu)=0$ be a bivariate polynomial system with polynomials of total degree $d$.
This system has generically $d^2$ roots, which are all finite.
The uniform determinantal representation of \cite{BDD17} produces
$(2d-1)\times (2d-1)$ matrices $A_i, B_i, C_i$ for $i=1,2$ such that
$p_i(\lambda,\mu) = \det(A_i+\lambda B_i+\mu C_i)$ for $i=1,2$,
and solutions $(\lambda,\mu)$ are eigenvalues of the singular two-parameter eigenvalue problem
\[
\begin{aligned}
(A_1+\lambda B_1+\mu C_1)\,x_1&=0,\cr
(A_2+\lambda B_2+\mu C_2)\,x_2&=0.
\end{aligned}
\]
For more details on (singular) two-parameter eigenvalue problems see, e.g.,
\cite{Atk72,MPl09,KoPl22}.
If we just want to compute the part $\lambda$ of solutions $(\lambda,\mu)$, this results in a
singular generalized eigenvalue problem $(\Delta_1-\lambda \Delta_0)\,z=0$, where
$\Delta_1 = C_1 \otimes A_2 - A_1 \otimes C_2$ and
$\Delta_0 = B_1 \otimes C_2 - C_1 \otimes B_2$ are of size $(2d-1)^2\times (2d-1)^2$.
For a computed $\lambda$ we can then compute $\mu$ from the
pencils $A_1+\lambda B_1+\mu C_1$ and $A_2+\lambda B_2+\mu C_2$; for details see, e.g., \cite[Algorithm~2]{HMP19}.

For a numerical example we consider a system of bivariate polynomials from \cite[Ex.~7.1]{HMP19}
\begin{align*}
p_1(\lambda,\mu) & = 1 + 2\lambda + 3\lambda + 4\lambda^2 +
5\lambda \mu + 6\mu^2 + 7\lambda^3 + 8\lambda^2\mu +
9\lambda\mu^2 + 10\mu^3 = 0,\\
p_2(\lambda,\mu) & = 10 + 9\lambda + 8\mu + 7\lambda^2 + 6\lambda\mu +
 5\mu^2 + 4\lambda^3 + 3\lambda^2\mu + 2\lambda\mu^2 + \mu^3 = 0.
\end{align*}
A uniform determinantal representation from \cite{BDD17} gives
a two-parameter eigenvalue problem of the form
\begin{equation}\label{ex:73poly}
\begin{aligned}
A_1+\lambda B_1+\mu C_1 &=\smath{\left[\begin{array}{ccccc}
 0 & 0 & 4 + 7\lambda & \ph-1 & \ph-0\\
 0 & 5 + 8\lambda & 2 & -\lambda & \ph-1 \\
 6 + 9\lambda + 10\mu & 3 & 1 & \ph-0 & -\lambda\\
 1 & -\mu & 0 & \ph-0 & \ph-0\\
 0 & 1 & -\mu & \ph-0 & \ph-0 \end{array}\right]}, \\[1mm]
A_2+\lambda B_2+\mu C_2 &=\smath{\left[\begin{array}{ccccc}
 0 & 0 & 7 + 4\lambda & \ph-1 & \ph-0\\
 0 & 6 + 3\lambda & 9 & -\lambda & \ph-1 \\
 5 + 2\lambda + \mu & 8 & 10 & \ph-0 & -\lambda\\
 1 & -\mu & 0 & \ph-0 & \ph-0\\
 0 & 1 & -\mu & \ph-0 & \ph-0 \end{array}\right]},
\end{aligned}
\end{equation}
where $p_i(\lambda,\mu)=\det(A_i+\lambda B_i+\mu C_i)$ for $i=1,2$.
There are $9$ solutions $(\lambda,\mu)$ and we can compute the $\lambda$-parts as eigenvalues of the
corresponding singular generalized eigenvalue problem $(\Delta_1-\lambda \Delta_0)\,z=0$ of size $25\times 25$ that
has normal rank $21$. \Cref{tab:ex73} contains the results obtained from Algorithm~1 and Algorithm~3. Compared to \cite[Ex.~7.1]{HMP19}, where we have solved the same problem with
the rank-completing algorithm that requires solutions of a problem of size $25\times 25$,
we now use the projection method that leads to a smaller generalized eigenvalue problem of size $21\times 21$.
The gap between finite and infinite eigenvalues can clearly be seen. Notice also that all eigenvalues of the
projected pencil are true eigenvalues. We show below that this holds for a generic system of polynomials
if we use the same uniform determinantal representation.

\begin{table}[htb!] \label{tab:ex73}
\centering
\caption{Results of Algorithm~1 (projected pencil) followed by Algorithm~3 applied to the pencil
$\Delta_1-\lambda \Delta_0$ related to \eqref{ex:73poly}.}
{\footnotesize \begin{tabular}{c|clllll} \hline \rule{0pt}{2.3ex}%
$j$ & $\lambda_j$ & $\quad \ \gamma_j$ & $\ \alpha_i$ & $\beta_j$ & ${\rm gap}_j$ & Type \\[0.5mm]
\hline \rule{0pt}{2.5ex}%
1--2 & $-1.133090 \pm 0.3011559i$ & $9.5\cdot 10^{-5}$ & $2.5\cdot 10^{-17}$ & $8.6\cdot 10^{-16}$ & $0.39$ & Finite true \\[-1.5mm]
\vdots & \vdots & \vdots & \vdots & \vdots & \vdots & \vdots \\
8--9 & $-0.5608503\pm2.035545i$ & $5.1\cdot 10^{-6}$ & $2.4\cdot 10^{-17}$ & $1.7\cdot 10^{-16}$ & $0.44$ & Finite true \\
10 & $-2556.290$ & $1.0\cdot 10^{-20}$ & $7.4\cdot 10^{-17}$ & $2.1\cdot 10^{-16}$ & $1.00$ & Infinite true \\[-1.5mm]
\vdots & \vdots & \vdots & \vdots & \vdots & \vdots & \vdots \\
21 & $\infty$ & $0.0$ & $5.9\cdot 10^{-17}$ & $5.5\cdot 10^{-17}$ & $1.00$ & Infinite true \\
\hline
\end{tabular}}
\end{table}

Some properties of singular pencils related to the determinantal representation of bivariate polynomials of total degree $d$ \cite{BDD17} are presented in \cref{tab:1}. The associated singular pencil has size $n = (2d-1)^2$ and rank-completion
index $k=(d-1)^2$. The perturbed pencil \eqref{pert} therefore has $k$ prescribed eigenvalues.
The remaining $n-k$ eigenvalues are all true, i.e., there are no random eigenvalues.
Of those $n-k$ eigenvalues, $d^2$ correspond to roots of the polynomial system,
while the remaining true eigenvalues are $\infty$ and have no meaning for the original problem.

\begin{table}[!htbp]
\begin{footnotesize}
\begin{center}
\caption{Properties of singular pencils coming from the uniform determinantal representations \cite{BDD17}
of bivariate polynomials of total degree $d$: size of the pencil $n$, number of true eigenvalues $n_{\rm f}+n_{\rm i}$, number of infinite eigenvalues $n_{\rm i}$, number of prescribed eigenvalues $k$.
Note that this type of application does not have random eigenvalues ($n_{\rm r}=0$). Cf.~also \cref{tab:nreigs}.}
\label{tab:1}
\begin{tabular}{cccccccc} \hline \rule{0pt}{2.1ex}%
$d$ & $n$ & \#True & \#$\infty$ & \#Roots & \#Random & \#Prescr.~($k$) & $k/n$ \\ \hline \rule{0pt}{2.3ex}%
\ph14 & \ph149 & \ph140 & \ph124 & \ph116 & 0 & \ph19 & 0.184 \\
\ph15 & \ph181 & \ph165 & \ph140 & \ph125 & 0 & 16 & 0.198 \\
\ph16 & 121    & \ph196 & \ph160 & \ph136 & 0 & 25 & 0.207 \\
\ph17 & 169    & 133    & \ph184 & \ph149 & 0 & 36 & 0.213 \\
10    & 361    & 280    & 180    & 100    & 0 & 81 & 0.224 \\
\hline
\end{tabular}
\end{center}
\end{footnotesize}
\end{table}

It turns out that $A_1$ (and hence $A_2$) is of full rank $2d-1$,
and $B_1$ and $C_1$ (and hence $B_2$ and $C_2$) are of rank $d$.
Then $C_1 \otimes A_2$ and $A_1 \otimes C_2$ are of rank $2d^2-d$.
Also, $\Delta_1 = C_1 \otimes A_2 - A_1 \otimes C_2$ has $(d-1)^2$ zero rows
and $(d-1)^2$ zero columns, and is, apart from this, of full rank $3d^2-2d$.
Therefore $k = (2d-1)^2 - (3d^2-2d) = (d-1)^2$.
This result means that we can reduce the size of the pencil substantially
using Algorithm~1. We do have to compute the matrix products $W^*\Delta_1Z$ and $W^*\Delta_0Z$,
but this computation, although of complexity $O(d^6)$, is computationally
negligible compared to the the task of finding the eigenvalues and the right and left
eigenvectors of the pencil $W^*\Delta_1Z-\lambda \, W^*\Delta_0Z$.

We have for the asymptotic ratio of singularity, for large total degree $d$,
\[
\frac{k}{n} = \frac{(d-1)^2}{(2d-1)^2} \ \to \ \frac{1}{4}.
\]
Also, the number of true eigenvalues that are roots of the polynomial system
converges to the same ratio:
\[
\frac{\#\text{roots}}{n} = \frac{d^2}{(2d-1)^2} \ \to \ \frac{1}{4}.
\]
Moreover, we see that, interestingly, nearly 50\% of the eigenvalues
($\frac{1}{2}(n-1)$) are true infinite eigenvalues, and there are no random eigenvalues in this type of problem.
This happens because the KCF of $\Delta_1-\lambda \Delta_0$ contains only right singular blocks $L_0$ and left singular blocks $L_0^\top$ (corresponding to the zero rows and columns).

Although the singular pencil $\Delta_1-\lambda \, \Delta_0$ looks perfect for Algorithm~1, it may
happen that the method fails to identify all $d^2$ true eigenvalues if some of the eigenvalues
$\lambda$ are very ill-conditioned. Also, the
infinite eigenvalues appear in large blocks that increase in size with $d$ which makes it
difficult to properly identify finite eigenvalues. A possible solution is
to use arithmetic in higher precision (for instance by using the Multiprecision Computing Toolbox \cite{Advanpix}).
The construction that leads from a system of two bivariate polynomials to a
singular pencil $\Delta_1-\lambda \Delta_0$ can be seen as a type of resultant method.
It appears that this approach has the same problems as other resultant methods for finding
roots of systems of bivariate polynomials, as explained in \cite{NofTown}.
\end{example}

\begin{example}\label{ex:doubleeig}\rm Given two matrices $A$ and $B$ of size $n$, in \cite[Ex.~6.4]{HMP19} we have shown how to apply the rank-completing method to find values $\lambda$ such that $A+\lambda B$ has a double eigenvalue.
The idea from \cite{MPl14} is to look for independent vectors $x$ and $y$ such that
\begin{equation}\label{eq:dblorig}
\begin{split}
(A + \lambda B - \mu I) \, x & = 0, \\
(A + \lambda B - \mu I)^2 \, y & = 0,
\end{split}
\end{equation}
which is a quadratic two-parameter eigenvalue problem.
We can linearize \cref{eq:dblorig} as a linear two-parameter eigenvalue problem and solve the associated
singular eigenvalue problem $(\Delta_1-\lambda \Delta_0)z=0$ of size $3n^2\times 3n^2$; for details see \cite{HMP19}. The normal rank of
this pencil is $3n^2-n$ and, contrary to the previous example, infinite eigenvalues, which have multiplicity $n^2$, appear in
blocks $N_1$. Infinite eigenvalues are thus not an issue and the rank-completing algorithm from \cite{HMP19} performs better than
the staircase method.

Now we will show how to compute the values $\lambda$ even more efficiently.
Generically, if $A+\lambda B$ has a double eigenvalue, it is non-semisimple. If we assume that all double eigenvalues are non-semisimple, then we can write the problem as a linear two-parameter eigenvalue problem
\begin{equation}
\begin{split}\label{eq:dbeig}
(A + \lambda B - \mu I) \, x & = 0, \\
\left(\left[\begin{matrix}\ph{-}A & 0\cr -I & A\end{matrix}\right] + \lambda \left[\begin{matrix}B & 0\cr 0 & B\end{matrix}\right] - \mu
\left[\begin{matrix}I & 0\cr 0 & I\end{matrix}\right]\right) \, \left[\begin{matrix}z \cr y\end{matrix}\right] & = 0.
\end{split}
\end{equation}
The second equation of \cref{eq:dbeig} reads $(A+\lambda B-\mu I)\,z=0$ and $(A+\lambda B-\mu I)\,y=z$, which
means that vectors $z$ and $y$ form a Jordan chain for the $J_2(\mu)$ block of matrix $A+\lambda B$.
The problem \cref{eq:dbeig} is associated to a singular generalized eigenvalue problem
$(\widetilde \Delta_1-\lambda \widetilde \Delta_0)\,w=0$,
where
\begin{equation}\label{eq:Deltadbl}
\widetilde \Delta_1=A \otimes \left[\begin{matrix}I & 0\cr 0 & I\end{matrix}\right]
-I\otimes \left[\begin{matrix}\ph{-}A & 0\cr -I & A\end{matrix}\right],\quad
\widetilde \Delta_0 = I\otimes \left[\begin{matrix}B & 0\cr 0 & B\end{matrix}\right] -
B\otimes \left[\begin{matrix}I & 0\cr 0 & I\end{matrix}\right]
\end{equation}
are of size $2n^2\times 2n^2$. The problem has $2n^2-n$ as its normal rank,
$n(n-1)$ finite eigenvalues that are solutions we are interested in,
and $n$ infinite eigenvalues, again in blocks $N_1$. An improvement over \cref{eq:dblorig} and \cite{HMP19} is that using Algorithm~1 we have to solve a generalized eigenproblem of size $2n^2-n$, while in \cite{HMP19} we had to solve a generalized eigenproblem of size $3n^2$. This enables us to solve the problem for even larger matrices. For random matrices $A$ and $B$ of size $n=20$, the approach from \cite{HMP19} uses initial matrices of size $1200\times 1200$ and runs for $6.58$s, while the new approach, where we apply Algorithm~1 to $800\times 800$ matrices \cref{eq:Deltadbl} and we have to solve a generalized eigenproblem of size $780\times 780$, requires just $1.72$s. In both cases we found all $380$ solutions without any problems. The regular part is well separated from the random eigenvalues and the same applies to the finite and infinite eigenvalues.
\end{example}

\begin{example}\label{ex:recteig}\rm All presented algorithms can also be applied to rectangular pencils, where
$A,B\in\C^{m,n}$ and $m\ne n$. The most suitable one is Algorithm~1, in particular when the difference between $m$ and $n$ is large, because
we have to add an appropriate number of zero rows or columns to make the pencil square
for Algorithm~2 or \cite[Algorithm~1]{HMP19}, which is not needed for Algorithm~1.

For an example we consider the problem \cite[Example 11]{EmaVD82} of computing transmission zeros of a control system of the form
\begin{align*}
\dot x&=Ax+Bu, \\
y&=Cx+Du,
\end{align*}
where
\[
A = \smath{\left[\begin{matrix}-2 &-6 & \ph{-}3 & -7 & \ph{-}6\cr
\ph{-}0 & -5 & \ph{-}4 & -4 & \ph{-}8\cr
\ph{-}0 & \ph{-}2 & \ph{-}0 & \ph{-}2 & -2 \cr
\ph{-}0 & \ph{-}6 & -3 & \ph{-}5 & -6\cr
\ph{-}0 & -2 & \ph{-}2 & -2 &\ph{-}5\end{matrix}\right]}, \quad
B = \smath{\left[\begin{matrix}-2 & \ph{-}7\cr
-8 & -5 \cr
-3 & \ph{-}0 \cr
\ph{-}1 & \ph{-}5\cr
-8 & \ph{-}0\end{matrix}\right]}, \quad
C = \smath{\left[\begin{matrix}0 &-1 & \ph{-}2 & -1 & -1\cr
1 & \ph{-}1 & \ph{-}1 & \ph{-}0 & -1\cr
0 & \ph{-}3 & -2 & \ph{-}3 & -1\end{matrix}\right]},
\]
and $D$ is a zero $2\times 3$ matrix.
We have to find the eigenvalues of the system pencil
\begin{equation}\label{eq:sys_mat}
S(\lambda)=
\left[\begin{matrix}\lambda I - A & B\cr -C & D\end{matrix}\right],
\end{equation}
which is rectangular due to a different number of input and output variables.
Algorithm~1 returns the values in
\cref{tab:ex75proj} and correctly identifies transmission zeros $3$ and $-4$.

\begin{table}[htb!] \label{tab:ex75proj}
\centering
\caption{Results of Algorithm~1 (projected pencil) followed by Algorithm~3 applied to the singular pencil \eqref{eq:sys_mat}.}
{\footnotesize \begin{tabular}{c|clllll} \hline \rule{0pt}{2.3ex}%
$j$ & $\lambda_j$ & $\quad \ \gamma_j$ & $\quad \ \alpha_i$ & $\quad \ \beta_j$ & ${\rm gap}_j$ & Type \\[0.5mm]
\hline \rule{0pt}{2.5ex}%
1 & $ 3.000000$ & $2.0\cdot 10^{-2}$ & $5.7\cdot 10^{-17}$ & $8.8\cdot 10^{-17}$ & $1.52$ & Finite true \\
2 & $-4.000000$ & $6.3\cdot 10^{-3}$ & $2.6\cdot 10^{-16}$ & $4.9\cdot 10^{-17}$ & $0.53$ & Finite true \\
3--4 & $-14.50574 \pm 1.114237\cdot 10^{8} i$ & $1.2\cdot 10^{-23}$ & $3.6\cdot 10^{-17}$ & $1.4\cdot 10^{-16}$ & $1.00$ & Infinite true \\
5 & $-2.021401\cdot 10^{16}$ & $2.9\cdot 10^{-33}$ & $3.6\cdot 10^{-17}$ & $1.4\cdot 10^{-16}$ & $1.00$ & Infinite true \\
6 & $\infty$ & $0.0$ & $1.4\cdot 10^{-17}$ & $6.9\cdot 10^{-17}$ & $1.00$ & Infinite true \\
7 & $-1.815461$ & $3.1\cdot 10^{-2}$ & $1.2\cdot 10^{-2}$ & $0.0$ & $1.05$ & Random left \\
\hline
\end{tabular}}
\end{table}
\end{example}

\section{Normal rank}\label{sec:nrank} For both methods presented in this paper as well as for the original rank-updating approach from \cite{HMP19} it is important to correctly determine the normal rank of an $n\times n$ singular pencil $A-\lambda B$. For some applications, for instance for the singular problem related to polynomial systems from \cref{sec:poly} as well as for the double eigenvalue problem in \cref{ex:doubleeig}, the normal rank is known in advance and this is not an issue.

If the normal rank is not known in advance, we can determine it by computing $\text{rank}(A +\eta B)$ for a random value $\eta\in \C$. To increase
the probability that the normal rank is accurate, we may repeat the computation several times. Although
the computation of rank has complexity ${\cal O}(n^3)$, which is technically the same order as the complexity to compute eigenvalues and right and left eigenvectors of a matrix pencil of the same size, the constant is much smaller, and in practice the computation of the rank is a negligible part of the method. For instance, for two random $1000\times 1000$ matrices $A$ and $B$ in Matlab, running
{\tt rank(A+randn*B)} takes $0.06$s while {\tt [X,D,Y]=eig(A,B)} takes $3.45$s.

Even with a repeated computation it might happen that we fail to determine the normal rank correctly. In that case it is more likely that the normal rank is underestimated than overestimated, as this may typically happen when $\eta$ is selected close to an eigenvalue of $A-\lambda B$. We will show that, fortunately, if the supplied normal rank is an underestimation, we can detect this from the numerical results. The following result from \cite{DD16lowrank} can be applied to the rank-completing algorithm from
\cite{HMP19}, where nonzero multiplicities in the theorem correspond to sizes of the Jordan blocks in case of a finite eigenvalue or to sizes of the infinite blocks in case of an infinite eigenvalue.

\begin{theorem}\label{thm:deteran_dopico}\cite[Thm.~4.3]{DD16lowrank}
Let $\lambda_0$ be an eigenvalue (finite or infinite) of a regular $n\times n$ complex matrix pencil $A-\lambda B$, with nonzero
multiplicities $0<d_1\le \cdots \le d_m$. Let $0<s<m$ be an integer, and denote by $\mathbb P_s$ the set of all $n\times n$ matrix
pencils with normal rank at most $s$. Then there is a generic set ${\cal G}$ in $\mathbb P_s$ such that for all
$A_1-\lambda B_1\in{\cal G}$, the partial multiplicities of the perturbed pencil $A+A_1-\lambda(B+B_1)$ at $\lambda_0$ are
$0<d_1\le\cdots\le d_{m-s}$.
\end{theorem}

If we apply the rank-completing algorithm from \cite{HMP19} and underestimate the normal rank, then the value $k$ that we use in the perturbation \eqref{pert} is
$k=n-{\rm nrank}(A,B)+s$ for $s>0$. We
can apply \cref{thm:deteran_dopico} in a way, that in \eqref{pert} we first update the singular pencil using a perturbation of
normal rank $n-{\rm nrank}(A,B)$ to a regular pencil, for which we
can apply the theorem, and then add a perturbation of normal rank $s$.

\begin{remark}\label{rem:sub_81}\rm We observe a behavior similar to \cref{thm:deteran_dopico} when we
project an $n\times n$ regular pencil $A-\lambda B$ to an $(n-s)\times (n-s)$ matrix pencil
 $\widetilde A-\lambda \widetilde B:=U^*\!AV-\lambda \, U^*BV$, where $U$ and $V$ are random $n\times (n-s)$ matrices with
orthonormal columns and $s\ge 1$. A brief explanation for the case of a finite eigenvalue is as follows; the case with
an infinite eigenvalue is similar.

If $\lambda_0$ is an eigenvalue of $A-\lambda B$ of geometric multiplicity $m>s$ with corresponding Jordan blocks $J_{d_1}(\lambda_0),\ldots,J_{d_m}(\lambda_0)$,
ordered by size so that $d_1\le \cdots \le d_m$, then
generically we have that ${\rm dim}\left({\rm Im}(V)\cap {\rm Ker}(A-\lambda_0 B)\right)=m-s$, therefore
$\lambda_0$ is an eigenvalue of $\widetilde A-\lambda \widetilde B$ of geometric multiplicity $m-s$.

To show that generically the KCF of $\widetilde A-\lambda \widetilde B$ contains
the smallest $m-s$ Jordan blocks,
we apply a similar argument as for the geometric multiplicity.
We know that for $r=1,\ldots,m$ the KCF of $A-\lambda B$ has exactly
$\dim\left({\rm Ker}\left((A-\lambda_0 B)^{r}\right)/ \, {\rm Ker}\left((A-\lambda_0 B)^{r-1}\right)\right)$ blocks
$J_d(\lambda_0)$ of size at least $r$. If we intersect these subspaces with
${\rm Im}(V)$, we see that if the KCF of $A-\lambda B$ has $m_r>s$ Jordan blocks
of size at least $r$, then the KCF of $\widetilde A-\lambda \widetilde B$
generically has $m_r-s$ Jordan blocks of size at least $r$ for $r=1,\ldots,m$.
Therefore, the KCF of $\widetilde A-\lambda \widetilde B$ generically contains
the Jordan blocks $J_{d_1}(\lambda_0),\ldots,J_{d_{m-s}}(\lambda_0)$.
\end{remark}

Suppose that the estimated normal rank is equal to ${\rm nrank}(A,B)-s$, where
$s>0$. In such case it follows from
\cref{thm:deteran_dopico} and \cref{rem:sub_81} that
all three methods generically return only multiple finite eigenvalues with a
geometric multiplicity $m>s$.
If $\lambda_0$ is an eigenvalue of $A-\lambda B$ with Jordan blocks
$J_{d_1}(\lambda_0),\ldots,J_{d_m}(\lambda_0)$, ordered so that
$d_1\le \cdots \le d_m$, then all three methods return $d_1+\cdots+d_{m-s}$
eigenvalues close to $\lambda_0$. The same applies to $\infty$ if there are more
than $s$ corresponding infinite blocks in the KCF. All remaining eigenvalues are
identified to be of the prescribed type, i.e., both $\alpha_i\ne 0$ and
$\beta_i\ne 0$, where
$\alpha_i,\beta_i$ are as defined in Algorithm~1 and Algorithm~2, and
$\alpha_i=\|V^*x_i\|$ and
$\beta_i=\|U^*y_i\|$ for \cite[Algorithm~1]{HMP19}.
For the later case, it is easy to see that generically, if
$\dim({\rm Im}(V)) > n - {\rm nrank}(A-\lambda B)$, then a right eigenvector $x$ of \eqref{pert} is
not
orthogonal to $V$ unless it is an
 eigenvector for an eigenvalue (finite or inifinite) of $A-\lambda B$ with geometric multiplicity larger than $s$,
 and similarly $U^*y\ne 0$ for a left eigenvector $y$.

Since we know that there should be no eigenvalues
of prescribed type in Algorithm~1, a presence of such eigenvalues
clearly indicates that the normal rank has been underestimated. We can apply this
also to \cite[Algorithm~1]{HMP19}, where in such case the number of eigenvalues
of the prescribed type will exceed $n-{\rm nrank}(A,B)$.

\begin{example}\label{ex:under}\rm If we apply Algorithm~1 to the singular pencil from \cref{ex:gaps}, but use
15 for a normal rank instead of the correct 16, we get the values in \cref{tab:underestimate}.

\begin{table}[htb!] \label{tab:underestimate}
\centering
\caption{Results of Algorithm~1 (projection to normal rank) followed by Algorithm~3 applied to the singular pencil from \cref{ex:gaps},
using an underestimated value 15 for the normal rank.}
{\footnotesize \begin{tabular}{r|clllll} \hline \rule{0pt}{2.3ex}%
$j$ & $\lambda_j$ & $\quad \ \gamma_j$ & $\quad \ \alpha_i$ & $\quad \ \beta_i$ & ${\rm gap}_j$ & Type \\[0.5mm]
\hline \rule{0pt}{2.5ex}%
1 & $1.000000$ & $2.1\cdot 10^{-5}$ & $1.5\cdot 10^{-17}$ & $2.5\cdot 10^{-16}$ & $1.7\cdot 10^{-8}$ & Finite true \\
2 & $1.000000$ & $1.9\cdot 10^{-10}$ & $8.8\cdot 10^{-18}$ & $2.7\cdot 10^{-16}$ & $1.7\cdot 10^{-8}$ & Finite true \\
3 & $1.000000$ & $1.9\cdot 10^{-10}$ & $1.7\cdot 10^{-17}$ & $2.9\cdot 10^{-16}$ & $1.7\cdot 10^{-8}$ & Finite true \\
4 & $\infty$ & $0.0$ & $1.8\cdot 10^{-16}$ & $1.8\cdot 10^{-16}$ & $1.00$ & Infinite true \\
5 & $-0.1883913$ & $7.7\cdot 10^{-3}$ & $1.6\cdot 10^{-2}$ & $2.7\cdot 10^{-2}$ & $2.7\cdot 10^{-2}$ & Prescribed\\
6 & $-0.1063261 + 0.8025023i$ & $3.4\cdot 10^{-3}$ & $7.5\cdot 10^{-3}$ & $1.1\cdot 10^{-3}$ & $0.58$ & Prescribed\\[-1mm]
$\vdots$ & $\vdots$ & $\vdots$ & $\vdots$ & $\vdots$ & $\vdots$ & $\vdots$\\
15 & $1.647482\cdot 10^1$ & $9.7\cdot 10^{-6}$ & $5.7\cdot 10^{-4}$ & $9.1\cdot 10^{-4}$ & $0.90$ & Prescribed\\
\hline
\end{tabular}}
\end{table}

We see a clear gap between eigenvalues identified as true eigenvalues and those identified as prescribed eigenvalues,
\[
\max_{j=1,\ldots,4}(\max(\alpha_j,\beta_j))=2.9\cdot 10^{-16} \ll \min_{j=5,\ldots,15}(\max(\alpha_j,\beta_j))=5.7\cdot 10^{-4}.
\]
The presence of eigenvalues of the prescribed type indicates that the estimated normal rank is incorrect. As predicted, the method
returned 3 instances of eigenvalue 1 and one infinite eigenvalue, because the regular part of the KCF of $A-\lambda B$ contains the blocks $J_1(1),J_2(1),J_4(1),N_1$, and $N_2$. If we use the same method with an estimate 14 for the normal rank, the method returns the eigenvalue $1$ and $13$ prescribed eigenvalues. For an estimate 13 or less for the normal rank, the method does not return any true eigenvalue.
\end{example}

What happens if we overestimate the normal rank? In this case the obtained
pencil (perturbed, projected, or augmented) is still singular and we
can expect similar problems if we apply for instance the QZ algorithm to the
original pencil.
In many cases the algorithm still works and the true eigenvalues are among the computed
ones, but it is not
as reliable as if we use the correct normal rank. In a similar way as we can
detect an underestimated rank from the presence of eigenvalues of prescribed
type, in the case of an overestimated normal rank we get $\alpha_i=0$ and
$\beta_i=0$ for all
eigenvalues and they are all identified as true eigenvalues. In the second phase, when we apply Algorithm~3 to extract the
finite eigenvalues based on their values of $\gamma_i$ and ${\rm gap}_i$, all eigenvalues that are not
finite true eigenvalues satisfy $\gamma_i=0$ and are declared as infinite eigenvalues. To demonstrate this, in the next example we apply Algorithm~1 to the same pencil as in \cref{ex:under}, only this time we overestimate the normal rank.

\begin{example}\label{ex:over}\rm If we apply Algorithm~1 to the singular pencil from \cref{ex:gaps}, but use
17 for a normal rank instead of correct 16, we get the values in \cref{tab:overestimate}.

\begin{table}[htb!] \label{tab:overestimate}
\centering
\caption{Results of Algorithm~1 (projection to normal rank) followed by Algorithm~3 applied to the singular pencil from \cref{ex:gaps},
using an overestimated value 17 for the normal rank.}
{\footnotesize \begin{tabular}{c|clllll} \hline \rule{0pt}{2.3ex}%
$j$ & $\lambda_j$ & $\quad \ \gamma_j$ & $\quad \ \alpha_i$ & $\quad \ \beta_i$ & ${\rm gap}_j$ & Type \\[0.5mm]
\hline \rule{0pt}{2.5ex}%
1 & $\ph{-}1.000000$ & $1.6\cdot 10^{-7}$ & $2.1\cdot 10^{-17}$ & $1.2\cdot 10^{-17}$ & $6.1\cdot 10^{-9}$ & Finite true \\
2 & $\ph{-}1.000000$ & $2.3\cdot 10^{-11}$ & $2.5\cdot 10^{-17}$ & $1.2\cdot 10^{-17}$ & $6.1\cdot 10^{-9}$ & Finite true \\
3 & $\ph{-}1.000000$ & $2.3\cdot 10^{-11}$ & $2.8\cdot 10^{-17}$ & $1.2\cdot 10^{-17}$ & $6.1\cdot 10^{-9}$ & Finite true \\
4--5 & $0.9998994 \pm 1.006233\cdot 10^{-4}i$ & $9.6\cdot 10^{-14}$ & $3.2\cdot 10^{-18}$ & $7.3\cdot 10^{-17}$ & $1.0\cdot 10^{-4}$ & Finite true \\
6--7 & $\ph{-}1.000101 \pm 1.006154\cdot 10^{-4}i$ & $9.6\cdot 10^{-14}$ & $5.1\cdot 10^{-18}$ & $7.0\cdot 10^{-17}$ & $1.0\cdot 10^{-4}$ & Finite true \\
8 & $-0.8711577$ & $9.8\cdot 10^{-17}$ & $1.4\cdot 10^{-17}$ & $2.3\cdot 10^{-17}$ & $0.38$ & Infinite true\\
9-10 & $-4.668463\cdot 10^{-3} \pm 0.2474057i$ & $8.8\cdot 10^{-18}$ & $4.2\cdot 10^{-17}$ & $5.4\cdot 10^{-17}$ & $0.43$ & Infinite true\\[-1.5mm]
$\vdots$ & $\vdots$ & $\vdots$ & $\vdots$ & $\vdots$ & $\vdots$ & $\vdots$\\
17 & $\infty$ & $0.0$ & $2.4\cdot 10^{-17}$ & $4.8\cdot 10^{-17}$ & $1.00$ & Infinite true\\
\hline
\end{tabular}}
\end{table}

All values $\alpha_i,\beta_i$ are smaller than $1.3\cdot 10^{-16}$ and all
eigenvalues are declared to be from the regular part. In this particular example, in the second phase instances of eigenvalue 1 that has multiplicity 7 are correctly identified as finite eigenvalues. \end{example}

The above examples together with the discussion show that the situation is not without hope if the computed normal rank does not equal the actual normal rank. In the case of an underestimation this can be detected via the presence of eigenvalues of prescribed type, which should not be present. In this case we can repeat the computation of the normal rank or manually increase the
estimate. In case of an overestimate, we lose the ability to separate eigenvalues into true and random ones based on values $\alpha_i$ and $\beta_i$, but even in this case it can happen that Algorithm~3 correctly identifies the finite eigenvalues.

\section{Proof of Theorem \ref{main}}\label{sec:proof}
In this section, we provide the proof for the main theoretic result stated in Theorem~\ref{main}. We start with some technical preparations.
\begin{definition}
Let $\Omega\subseteq\C^{n,k}$. Then we call
$\Omega^\perp:=\{W\in \C^{n,n-k}\ |\ \exists U\in\Omega, \, W^*U=0\}$
the \emph{pointwise orthogonal complement} of $\Omega$ in $\C^{n,n}$.
\end{definition}
We aim to show that with $\Omega$ also
$\Omega^\perp$ is generic.

\begin{remark}\label{qlemma} \rm
For the proof of \cref{prop:perpgen}, we will need auxiliary
functions $q_j:(\C^n)^j\to\C^n$, $j=1,\dots,\ell$, where $\ell \le n$, with the following properties:
\begin{enumerate}
\item[1)] $\langle q_i(w_1,\dots,w_i), \, q_j(w_1,\dots,w_j)\rangle=0$ for $i,j=1,\dots,\ell$, $i\ne j$, and for all $w_1,\dots,w_\ell\in \C^n$.
\item[2)] $\langle q_i(w_1,\dots,w_i), \, w_j\rangle=0$ for $i=1,\dots,\ell$ and $j<i$.
\item[3)] The entries of $q_j(w_1,\dots,w_j)$ are polynomials in the real and imaginary parts of the entries of $w_1,\dots,w_j\in \C^n$ for $j=1,\dots,\ell$.
\end{enumerate}
These function can easily be recursively constructed using a Gram--Schmidt type orthogonalization without a normalization step. The omission of this step guarantees that the entries of the orthogonalized vectors are polynomials in the real and imaginary parts of the entries of the original vectors $w_1,\dots,w_\ell$. Indeed, set
$q_1(w_1):=w_1$ and if $q_1(w_1),\dots,q_{\ell-1}(w_1,\dots,w_{\ell-1})$ have already been constructed then define
\[
q_\ell(w_1,\dots,w_\ell):=\bigg(\prod_{j=1}^{\ell-1}\langle q_j,q_j\rangle\bigg) w_\ell-\sum_{i=1}^{\ell-1}\bigg(\prod_{j\ne i}\langle
q_j,q_j\rangle\bigg) \, \langle q_i,w_\ell\rangle \, q_i,
\]
where for the ease of notation we have dropped the dependence of $q_1,\dots,q_{\ell-1}$ on $w_1,\dots,w_{\ell-1}$.
The verification of the properties $1)$--$3)$ is then straightforward.
\end{remark}

\begin{proposition}\label{prop:perpgen}
Let $\Omega\subseteq\C^{n,k}$ be a generic set. Then also $\Omega^\perp\subseteq\C^{n,n-k}$
is a generic set.
\end{proposition}

\begin{proof} We take an arbitrary $\widehat U= [\widehat u_1 \ \, \dots \ \, \widehat u_k]\in\Omega$.
By definition the pointwise orthogonal complement of $\Omega$ is contained in the set of common zeros of finitely many
polynomials in $2nk$ real variables. Since $\widehat U\in\Omega$, there exists a polynomial $p$ out of those polynomials
such that $p(\widehat U)\ne 0$.
Furthermore, let $q_1,\dots,q_{n-k+1}$ be the functions from \cref{qlemma} and define
\[
\widehat p(w_1,\dots,w_{n-k})=p\big(q_{n-k+1}(w_1,\dots,w_{n-k},\widehat u_{1}), \
\dots, \ q_{n-k+1}(w_1,\dots,w_{n-k},\widehat u_k)\big).
\]
Then it follows from \cref{qlemma} that $\widehat p$ is a polynomial in the
real and imaginary parts of the $n(n-k)$ entries of $w_1,\dots,w_{n-k}$.
(In fact, it is also a polynomial in the real and imaginary parts of the
entries of $w_1,\dots,w_{n-k},\widehat u_{1},\dots,\widehat u_k$, but we regard the
entries of $\widehat u_{1},\dots,\widehat u_k$
as constants here.) Furthermore, $\widehat p$ is not the zero polynomial. Indeed, let
$\widehat W = [\widehat w_1 \ \, \dots \ \, \widehat w_{n-k}] \in \C^{n,n-k}$ be a matrix with
orthonormal columns such that $\widehat W^*\widehat U=0$.
Then we have
$q_{n-k+1}(\widehat w_1,\dots,\widehat w_{n-k},\widehat u_i)=\widehat u_i$ for $i=1,\dots,k$
and consequently that $\widehat p(\widehat w_1,\dots,\widehat w_{n-k})=p(\widehat u_1,\dots,\widehat u_k)\ne 0$.

Now let $W = [w_1 \ \, \dots \ \, w_{n-k}]$ be such that
$\widehat p(w_1,\dots,w_{n-k})\ne 0$ and set
\[
U := [\, q_{n-k+1}(w_1,\dots,w_{n-k},\widehat u_1) \ \ \dots \ \ q_{n-k+1}(w_1,\dots,w_{n-k}, \widehat u_k) \, ].
\]
Then by construction and by definition of $\widehat p$ we have $W^*U=0$ and $U\in\Omega$ which implies $W\in\Omega^\perp$.
It follows by contraposition that the complement of $\Omega^\perp$ is contained in the set of zeros of $\widehat p$ and thus $\Omega^\perp$ is a generic set.
\end{proof}

We are now able to prove the main theoretical result of this paper.

\begin{proof} (of Theorem~\ref{main}).
First, we claim that there exists a generic set $\Omega_r\subseteq\C^{n,n-k}$ with the property that for each $W\in\Omega_r$
there exists a generic set $\Omega_r'\subseteq\C^{n,n-k}$ such that for all $\widehat Z\in\Omega_r'$
the pencil $W^*(A-\lambda B)\,Z$ is regular, where $Z$ is the upper $m\times(n-k)$ submatrix of $\widehat Z$ as in~\eqref{eq:11.8.19}.
To see this, observe that the determinant of $W^*(A-\lambda B)\,Z$ is a polynomial
in the real and imaginary parts of the entries of $W$ and $Z$ (and thus also in the entries of $W$ and $\widehat Z$)
that is not the zero polynomial, because the pencil $A-\lambda B$ has normal rank $n-k$ which means that
there exists a nonzero minor of size $n-k$. (Choosing $W$ and $\widehat Z$ appropriately, this minor can be extracted giving
a nonzero value for the determinant of $W^*(A-\lambda B)Z$.) Applying \cref{lem:gen} then proves the existence
of $\Omega_r$ and $\Omega_r'$ as above, where the latter set depends on $W\in\Omega_r$.

In the following, we will frequently work with the extended $n\times n$ pencil $\widehat A-\lambda\widehat B$ that is obtained from
$A-\lambda B$ by just adding $n-m$ zero columns.

Next, let $\Omega:=(\Omega_1\cap\Omega_2)^\perp\cap\Omega_r\subseteq\C^{n,n-k}$, where $\Omega_1$ and $\Omega_2$ are the
generic subsets of $\C^{n,k}$ from \cref{prop4.2} applied to $A-\lambda B$ and \cref{thm:nfp} applied
to $\widehat A-\lambda\widehat B$. Then $\Omega$ is generic.
Let $W\in\Omega$ be fixed. Then there exists a matrix $U\in\Omega_1\cap\Omega_2$ such that $W^*U=0$.
Since in particular $U\in\Omega_1$, there exist nonsingular matrices $P\in \C^{n,n}$ and $Q\in \C^{m,m}$ such that
\begin{equation}\label{eq:7.8.19}
P\,(A-\lambda B)\,Q=\left[\begin{array}{cc}R(\lambda)&0\\ 0& S(\lambda)\end{array}\right]\quad\mbox{and}\quad
PU=\left[\begin{array}{c}0\\ \widetilde U\end{array}\right],
\end{equation}
where $R(\lambda)$ and $S(\lambda)$ are the regular and singular parts of $A-\lambda B$, respectively, and
$PU$ is partitioned conformably with $P\,(A-\lambda B)\,Q$. Setting
$\widehat Q:=\operatorname{diag}(Q,I_{n-m})$, we then obtain
\begin{equation}\label{eq:8.8.19}
P\,(\widehat A-\lambda \widehat B) \, \widehat Q=\left[\begin{array}{ccc}R(\lambda)&0&0\\
0& S(\lambda)&0\end{array}\right]\quad\mbox{and}\quad
PU=\left[\begin{array}{c}0\\ \widetilde U\end{array}\right].
\end{equation}
Let $R(\lambda)$ have the size $r\times r$, i.e., we have
$\widetilde U\in \C^{n-r,k}$. Since $U$ and thus also $\widetilde U$ have full column rank $k$, there exists
a nonsingular matrix $T\in \C^{n-r,n-r}$ such that with $\widetilde P:=\operatorname{diag}(I_r,T)\cdot P$ we have
\[
\widetilde P\,(A-\lambda B)\,Q=\left[\begin{array}{ccc}R(\lambda)&0&0\\ 0&S_{11}(\lambda)&S_{12}(\lambda)\\ 0&S_{21}(\lambda)&S_{22}(\lambda)
\end{array}\right]\quad\mbox{and}\quad \widetilde PU=\left[\begin{array}{c}0\\ 0\\ I_k\end{array}\right],
\]
where $S_{11}(\lambda)\in \C^{n-k-r,n-k-r}$ and $\widetilde PU$ is partitioned conformably with $\widetilde P\,(A-\lambda B)\,Q$.
Observe that $W^*U=0$ implies $W^*\widetilde P^{-1}=[W_1 \ \, W_2 \ \, 0]$,
where $W_1\in \C^{n-k,r}$ and $W_2\in \C^{n-k,n-k-r}$ both have full column rank, i.e., the matrix
$\widetilde W:=[W_1 \ \, W_2]$ is invertible.

Next, we claim that there exists a generic set $\Omega_0\subseteq\C^{n,n-k}$ such that for each $\widehat Z\in\Omega_0$
the first $n-k$ rows of $\widehat Q^{-1}\widehat Z$ are linearly independent. This easily follows since the determinant of this submatrix
is a nonzero polynomial in the real and imaginary parts of the entries of $\widehat Z$. (Note that $\Omega_0$ depends on $\widehat Q$ and hence on $W$.)
Define $\Omega':=\Omega_0\cap(\Omega_3^\perp)\cap\Omega_r'\subseteq\C^{n,n-k}$, where $\Omega_r'$ and $\Omega_3$, the generic subset
from \cref{thm:nfp} (applied to the extended pencil $\widehat A-\lambda\widehat B$), depend on $W$.
Then $\Omega'$ is generic. Let $\widehat Z\in\Omega'$ and partition $\widehat Z$ as in~\eqref{eq:11.8.19}, i.e.,
\[
\widehat Z=\left[\begin{array}{c}Z\\ Z'\end{array}\right],\quad\mbox{where}\quad Z\in \C^{m,n-k},\ Z'\in \C^{n-m,n-k},
\]
and further partition
\[
Q^{-1}Z=\left[\begin{array}{c}\widetilde Z\\ Z_3\end{array}\right],\quad\mbox{where}\quad
\widetilde Z=\left[\begin{array}{c} Z_1\\ Z_2\end{array}\right],\
Z_1\in \C^{r,n-k},\ Z_2\in \C^{n-k-r,n-k},\ Z_3\in \C^{m-n+k,n-k}.
\]
Observe that due to the special block structure $\widehat Q=\operatorname{diag}(Q,I_{n-m})$, we have
\[
\widehat Q^{-1}\widehat Z=\left[\begin{array}{c}Q^{-1}Z\\ Z'\end{array}\right].
\]
Since $\Omega'\subseteq\Omega_0$ it follows that $\widetilde Z\in \C^{n-k,n-k}$ is
nonsingular and since $\Omega'\subseteq\Omega_3^\perp$ there exists a matrix $V\in \C^{n,k}$ such that $\widehat Z^*V=0$.
We then obtain
\[
W^*(A-\lambda B)Z=\widetilde W
\left[\begin{array}{c}R(\lambda)Z_1\\ S_{11}(\lambda)Z_2+S_{12}(\lambda)Z_3\end{array}\right]
=\widetilde W
\left[\begin{array}{cc}R(\lambda)&0\\ 0&\widetilde R(\lambda)\end{array}\right]
\widetilde Z,
\]
where $\widetilde R(\lambda)=S_{11}(\lambda)+S_{12}(\lambda)Z_3(Z_2^*Z_2)^{-1}Z_2^*$. Since
$\widetilde W$ and $\widetilde Z$ are invertible,
it follows that the pencil $W^*(A-\lambda B)Z$ is equivalent to the pencil $\operatorname{diag}(R(\lambda),\widetilde R(\lambda))$.
Furthermore, $U\in\Omega_2$ and $V\in\Omega_3$ together imply that
$\widehat A-\lambda \widehat B+\tau(UD_AV^*-\lambda UD_BV^*)$ is regular
and has the Kronecker canonical form
\[
\left[\begin{array}{ccc}R(\lambda)&0&0\\ 0&R_{\rm pre}(\lambda)&0\\ 0&0&R_{\rm ran}(\lambda)\end{array}\right]
\]
where $\tau\ne 0$, $D_A$, $D_B$, $R_{\rm pre}(\lambda)$ and $R_{\rm ran}(\lambda)$ are as in \cref{thm:nfp}. We will now show that $\widetilde R(\lambda)$
and $R_{\rm ran}(\lambda)$ are equivalent by showing that each eigenvalue of $R_{\rm ran}(\lambda)$ is also an eigenvalue of
$\widetilde R(\lambda)$. Since all eigenvalues of $R_{\rm ran}(\lambda)$ are simple and $\widetilde R(\lambda)$ has the same size
as $R_{\rm ran}(\lambda)$, it follows that the two pencils are equivalent. Now let
$\lambda_0\in \C\cup\{\infty\}$ be an eigenvalue of $R_{\rm ran}(\lambda)$ (generically, this eigenvalue will be finite)
and let $x$ be a right eigenvector and $y$ be a left eigenvector of $\widehat A-\lambda \,\widehat B+\tau\,(UD_AV^*-\lambda \, UD_BV^*)$
associated with $\lambda_0$. It follows from \cref{thm:nfp} that then either $V^*x=0$ or $U^*y=0$
(exactly one of these statements is true) and thus, since the columns of $W$ and $\widehat Z$ are bases of the orthogonal
complements of the ranges of $U$ and $V$, respectively, we obtain that either
$x=\widehat Z z$ for a nonzero $z\in \C^{n-k}$ or $y=Ww$ for a nonzero $w\in \C^{n-k}$ (again, exactly one of
the statements is true). Then either $w^*W^*(A-\lambda_0 B)Z=0$ or $W^*(A-\lambda_0 B)Zz=W^*(\widehat A-\lambda_0\widehat B)\widehat Zz=0$
and therefore $\lambda_0$ is an eigenvalue of $W^*(A-\lambda B)Z$. (Here, we have used that
$(A-\lambda B)Z=(\widehat A-\lambda \widehat B)\widehat Z$ which follows from the fact that $\widehat A-\lambda\widehat B$ has been obtained from
$A-\lambda B$ by just adding $n-m$ zero columns.) This finishes the proof of $1)$.

For the proof of $2)$, first observe that if $\lambda_0$ is an eigenvalue of $R(\lambda)$ and $x$ and $y$ are right
and left eigenvectors of the extended pencil $\widehat A-\lambda \widehat B$ associated with $\lambda_0$, respectively, then
following the same argument as in the paragraph above, it follows from \cref{thm:nfp} that both $x=\widehat Zz$ and
$y=Ww$ for nonzero vectors $z,w\in \C^{n-k}$. In particular, the map $x\mapsto z$ is a bijection from the set of right eigenvectors
of $\widetilde A-\lambda\widetilde B$ associated with $\lambda_0$ to the set of right eigenvectors of $W^*(A-\lambda B)Z$ associated
with $\lambda_0$. Similar observations hold for the map $y\mapsto w$, and even for the analogous maps in the case that $\lambda_0$ is
an eigenvalue of $R_{\rm ran}(\lambda)$, where, of course, the sets of eigenvectors have to be restricted to those eigenvectors
satisfying the orthogonality conditions from \cref{thm:nfp}.

Thus, if $\lambda_0\in \C$ is an eigenvalue of $W^*(A-\lambda B)Z$ with left eigenvector $y$ and right eigenvector $x$,
then both $Wy$ and $\widehat Zx$ are left respectively right eigenvectors of $\widehat A-\lambda\widehat B$ associated with $\lambda_0$ if
$\lambda_0$ is an eigenvalue of $R(\lambda)$. This implies $W_\perp(\widehat A-\lambda_0\widehat B)\widehat Zx=W_\perp(A-\lambda_0 B)Zx=0$
and $y^*W^*(\widehat A-\lambda_0\widehat B)Z_\perp=0$.
If on the other hand $\lambda_0$ is an eigenvalue of
$R_{\rm ran}$, then either $Wy$ or $\widehat Zx$ is a left or right eigenvector of $\widehat A-\lambda\widehat B$, respectively,
but not both. First assume, that $Wy$ is a left eigenvector of $\widehat A-\lambda\widehat B$ associated with $\lambda_0$. Then
$W_\perp^*(A-\lambda_0 B)Zx=W_\perp^*(\widehat A-\lambda_0\widehat B)\widehat Zx\ne 0$, because otherwise, keeping in mind
that $[W \ \, W_\perp]$ is nonsingular, we would have
$(\widehat A-\lambda_0\widehat B)\widehat Zx=0$ implying that $\widehat Zx$ is a right eigenvector which is a contradiction.
Analogously, we show that $y^*W^*(A-\lambda_0 B)Z_\perp \ne 0$ when $Zx$ is a right eigenvector of $\widehat A-\lambda\widehat B$
associated with $\lambda_0$. The claim in the case that $\lambda_0=\infty$ is an eigenvalue of $W^*(A-\lambda B)Z$
follows in a similar way.
\end{proof}

\section{Conclusions}
\label{sec:concl}
For the computation of eigenvalues of singular pencils, we have presented two new alternative approaches to the rank-complete perturbation
method \eqref{pert} of \cite{HMP19}.
Based on theory developed in \cref{sec:proj}, we have proposed a projection approach (Algorithm~1 in \cref{subsec:proj}), which reduces the
dimension of the problem, and may be particularly attractive when the normal rank $n-k$ is small.
While the perturbation approach from \cite{HMP19} usually gives very good results already, the projection method seems to be even more attractive due to its smaller size: while the accuracy of the two methods is usually comparable (and excellent), the projection scheme is computationally more efficient (depending on the value of $k$).
Compared to methods with perturbations of full rank (such as \cite{KreGlib22}), our methods need the normal rank of the pencil, but render more accurate results.

In \cref{sec:augm} we have proposed an augmentation approach (Algorithm~2), which enlarges
the dimension of the problem, but does not require any computations prior to the
solution of the generalized eigenproblem and does not change
the original matrices $A$ and $B$.
In \cref{sec:classif} we have studied in detail how to extract finite eigenvalues from eigenvalues of the regular part, leading to Algorithm~3.

For all three extraction techniques (perturbation, projection, augmentation), it is relevant to determine a correct value of the normal rank of the pencil.
In the total algorithm, this is an inexpensive operation.
Since it may occasionally be possible that the computed normal rank is incorrect, we have analyzed detection and repair of this situation in \cref{sec:nrank}.
In addition, for several singular pencils, it is easy to determine the normal rank
by inspecting the number of zero columns and rows
(provided that the remaining columns and rows are of full rank, which is often obvious from the matrix structure).
The polynomial systems of \cref{sec:poly} and the double eigenvalue problem of \cref{ex:doubleeig} are examples of this situation.

Code for the approaches developed in this paper is available in \cite{MultiParEig}.

\end{document}